\documentclass[gdweb]{geradwp}

\PassOptionsToPackage{hyphens}{url}

\usepackage[french]{babel}
\usepackage{hyperref}
\usepackage{tikz}
\usetikzlibrary{arrows.meta, positioning}
\usepackage{tikz}
\usepackage{amsmath}
\usepackage{amssymb}
\usepackage{amsthm}  
\usepackage{url}
\usepackage{placeins}
\usepackage{verbatim}
\usepackage{hyperref}
\usepackage{algorithm}
\usepackage{algpseudocode}
\usetikzlibrary{positioning, shapes.geometric, arrows.meta, decorations.pathreplacing}
\usepackage[sort,numbers]{natbib}
\usepackage{amsthm}
\theoremstyle{plain}

\newtheorem{theorem}{\sffamily Theorem}
\newtheorem{proposition}{\sffamily Proposition}

\theoremstyle{definition}

\newtheorem{assumption}{\sffamily Assumption}

\theoremstyle{remark}

\newcolumntype{C}[1]{>{\centering\arraybackslash}p{#1}}

\usepackage{tikz}
\usepackage{rotating}
\usepackage{adjustbox}
\usepackage{makecell}
\usepackage{multicol}
\usepackage{amsmath}
\usepackage{amssymb}
\usetikzlibrary{positioning,chains,fit,shapes,calc,arrows.meta,                        
	  backgrounds,                        
	  ext.paths.ortho,                    
	  ext.positioning-plus,               
	  }
\usepackage{tikz}
\usepackage{tkz-graph}
\usepackage{tkz-berge}
\usepackage{float}
\usetikzlibrary{calc}
\usetikzlibrary{decorations.pathreplacing}
\usetikzlibrary{arrows.meta}
\usepackage{footmisc}
\usepackage{pgfplots}
\pgfplotsset{compat=1.18}
\usetikzlibrary{spy}
\usepackage{newtxtext,newtxmath}
\usepackage{svg}
\usepackage{fontawesome}
\usepackage{amsthm}
\usepackage{mathtools}
\usepackage{enumitem}
\usepackage{soul}
\usepackage{graphicx}
\usepackage{longtable}
\usepackage{graphicx}   
\usepackage{tikz}
\usetikzlibrary{arrows.meta}
\usetikzlibrary{shapes.geometric, arrows}
\usepackage{tikz}
\usepackage{tkz-graph}
\usetikzlibrary{decorations.pathreplacing}

\definecolor{grayA}{RGB}{90,90,90}
\definecolor{grayB}{RGB}{140,140,140}
\definecolor{grayC}{RGB}{180,180,180}
\definecolor{grayD}{RGB}{220,220,220}
\definecolor{lightblue}{RGB}{120,170,255}
\definecolor{darkblue}{RGB}{5,15,80}
\soulregister\cite7
\soulregister\ref7
\soulregister\pageref7
\tikzstyle{startstop} = [rectangle, rounded corners, minimum width=3cm, minimum height=1.2cm, text centered, draw=black, fill=gray!50]
\tikzstyle{process} = [rectangle, minimum width=4cm, minimum height=1.2cm, text centered, draw=black, fill=white]
\tikzstyle{repeated} = [rectangle, minimum width=4cm, minimum height=1.2cm, text centered, draw=black, fill=gray!20, dashed]
\tikzstyle{decision} = [diamond, minimum width=3cm, minimum height=1.5cm, text centered, draw=black, fill=gray!30]
\tikzstyle{arrow} = [thick,->,>=stealth]

\GDcoverpagewhitespace{6.8cm}
\graphicspath{{Figures/}} 
\hypersetup{colorlinks,%
citecolor={blue}, 
urlcolor={blue},
linkcolor={blue},
breaklinks={true}
}

\makeatletter
\ifthenelse{\isundefined{\ALG@name}}{}%
{%
\renewcommand{\ALG@name}{\sffamily\footnotesize Algorithm}
}
\makeatother
\ifthenelse{\isundefined{\SetAlCapNameFnt}}{}%
{%
\SetAlCapNameFnt{\footnotesize}
\SetAlCapFnt{\sffamily\footnotesize}
}

\GDauthorsShort{A.~Abdellaoui, L.~Benabbou, I.~El~Hallaoui}
\GDauthorsCopyright{Abdellaoui, El~Hallaoui, Benabbou}
\GDpostpubcitation{Hamel, Benoit, Karine H\'ebert (2024). 
``Un exemple de citation'', \emph{Journal of Journals}, 
vol. X issue Y, p. n-m}{https://www.gerad.ca/fr}
\GDsupplementname{Internet Appendix}
\GDrevised{Mai}{May}{2024}

\title{\textbf{Learning to recover: Adaptive local branching with reinforcement learning for log-truck routing and scheduling under disruptions}}

\author{
Abdelhakim Abdellaoui$^{1,3,}$\thanks{Corresponding author: abdelhakim.abdellaoui@polymtl.ca} \quad
Loubna Benabbou$^{2}$ \quad
Issmail El Hallaoui$^{1}$ 
 \\
 Fran\c{c}ois Aub\'e$^{3}$\quad
Mouloud Amazouz$^{3}$
\\[1em]
{\small \textit{$^{1}$Department of Mathematics and Industrial Engineering, Polytechnique Montr\'eal, succ. Centre-ville, Montr\'eal, Qu\'ebec, H3C 3A7, Canada}} \\
{\small \textit{$^{2}$Department of Management Science, Universit\'e du Qu\'ebec \`a Rimouski (UQAR), Levis, QC G6V 0A6, Qu\'ebec, Canada}}\\
{\small\textit{$^{3}$Natural Resources Canada, CanmetENERGY, 1615 Lionel-Boulet Blvd, P.O. Box 4800, Varennes, Qu\'ebec, J3X 1P7, Canada}} 
}
\begin{document}
\date{}
\maketitle

\begin{GDabstract}{Abstract}
We consider the problem of reoptimising log-truck routing and scheduling 
operations in real time following unforeseen disruptions in the Canadian 
forestry industry. Such disruptions, including road closures, vehicle 
breakdowns, travel delays, and demand fluctuations, invalidate pre-established 
tactical plans and require recovery decisions that simultaneously restore 
operational feasibility and limit deviation from the original schedule, two 
objectives that are inherently in tension under tight time constraints.

We formulate the recovery problem as a sequence of neighbourhood-restricted 
mixed-integer linear programs parameterised by a deviation bound $k$ measuring 
the $\ell_1$ distance between the recovered and the baseline tactical plan. We show that the minimum feasible value $k_{\min}$ can 
be computed exactly, anchoring the recovery search at its most stable extreme. 
Rather than targeting a single value of $k$, we seek a Pareto front of 
non-dominated recovery plans spanning the stability--cost space, so as to 
present dispatchers with a structured set of operational options. To navigate 
this space efficiently, we propose a reinforcement learning policy, trained 
via the REINFORCE policy-gradient algorithm, that selects successive values 
of $k$ adaptively based on solver feedback, concentrating computational effort 
in productive regions of the search space and terminating exploration when 
further improvement is unlikely.

We evaluate the proposed framework on weekly instances derived from 
historical operational data provided by a Canadian forestry partner, under disruption scenarios covering all event categories 
considered. The reinforcement learning-guided approach recovers richer Pareto 
fronts with fewer solver calls than a fixed-$k$ grid and dichotomic search baselines, and produces 
feasible recovery plans within operationally acceptable reoptimization times 
across all tested configurations and disruption types.

\end{GDabstract}

\section{Introduction}

The forestry industry plays a vital role in the Canadian economy, contributing
significantly to national GDP~\cite{NRCan2022}. In recent years, however, the 
sector has faced increasing challenges due to global economic fluctuations, 
the COVID-19 pandemic, trade tensions such as U.S. tariffs, and intensified 
international competition. In this context, transportation costs, which 
represent up to 40\% of total operating costs~\cite{audy2013virtual}, have 
become a critical lever for maintaining profitability. Ensuring efficient 
and sustainable transportation logistics is therefore central to the 
industry's competitiveness, particularly in Canada, where vast geographic 
distances, seasonal constraints, and fluctuating market demands create unique 
operational challenges~\cite{abdellaoui2025decomposition}. Among these 
logistics challenges, log transportation stands out as a key research area, 
directly influencing operational efficiency and cost performance.

Planning in this domain is typically divided into tactical planning, which
defines medium-term strategies such as weekly or monthly plans, and operational
planning, which focuses on the day-to-day execution of these strategies.
Tactical planning aims to optimize overall efficiency based on static parameters
such as the road network, required demand, and known product supply, and 
operational planning is, in theory, expected to follow the tactical plan 
closely. In practice, however, forestry operations are exposed to frequent 
and often unpredictable disruptions~\cite{audy2012}, including sudden road 
closures due to weather events, equipment breakdowns, and changes in forest 
accessibility~\cite{amrouss2017}. In the Canadian context, our industrial 
partners emphasised that disruption is the norm rather than the exception: 
virtually every operational week involves at least one event that invalidates 
part of the tactical plan, and the absence of recovery tooling is the single 
most frequently cited reason for the limited operational adoption of existing 
log-truck planning systems.

Recovering from such disruptions requires producing a new transportation 
plan that satisfies all existing operational constraints together with those 
imposed by the event, while controlling cost and remaining as close as 
possible to the baseline plan. Through extensive discussions with our 
industrial partner, a sharper requirement emerged: dispatchers do not want a 
single prescribed recovery plan, but a structured set of trade-off options 
spanning the stability--cost space, from which the most suitable plan can be 
selected depending on the severity of the disruption and on operational 
priorities. Recovery is therefore inherently a bi-objective problem, and to 
the best of our knowledge no prior work has addressed the reoptimization of 
log-truck routing and scheduling under unforeseen events from this perspective.

Methodologically, the recovery problem sits at the intersection of three 
streams that have evolved largely in isolation. Local 
branching~\cite{fischetti2003local} provides an effective matheuristic for 
repair-type mixed-integer programs through a Hamming-distance constraint 
parameterised by a neighbourhood radius $k$, and \cite{liu2022learning} 
showed that this radius can itself be controlled by a reinforcement learning 
(RL) agent. In a parallel direction, multi-objective reinforcement learning 
(MORL)~\cite{roijers2013,vanmoffaert2014} has matured into a recognised 
paradigm for approximating Pareto fronts, but assumes cheap policy 
evaluations rather than the solver-call-bounded regime of MILP-based 
recovery. No existing work bridges these streams; nor does the local-branching 
literature provide analytical structure on the radius itself, even though 
this radius governs the entire feasibility--cost trade-off in repair settings.

\medskip

The contributions of this paper are as follows.
\begin{enumerate}
    \item We extend RL-guided local branching to a multi-objective setting, 
    in which the agent navigates a Pareto front rather than improving a 
    single incumbent, with a state, reward, and termination signal redesigned 
    around Pareto dominance.
    
    \item We derive analytical lower and upper bounds on the minimum feasible 
    deviation $k_{\min}$ that restores feasibility after a disruption, 
    independent of the RL machinery and applicable to any repair-type local 
    branching problem.
    
    \item We integrate these components into a Pareto-recovery framework that 
    returns a structured menu of non-dominated recovery plans, anchored at 
    $k_{\min}$ and explored adaptively under a fixed solver-call budget.
    
    \item We provide, to the best of our knowledge, the first reoptimization 
    framework for the Log-Truck Routing and Scheduling Problem under 
    unforeseen events, validated on twenty weekly instances from a Canadian 
    forestry partner and confirmed by partner dispatchers as addressing a 
    long-standing gap in deployed planning tools.
\end{enumerate}

The remainder of this paper is organised as follows. Section~\ref{litreview}
reviews the relevant literature on local branching, RL-guided matheuristics,
multi-objective RL, and disruption recovery in vehicle routing. 
Section~\ref{probdesc} describes the problem and presents the MILP 
formulation of the $\mathcal{LTRSP}$. Section~\ref{sec:methodology} details 
the recovery framework, including the analytical bounds on $k_{\min}$ and 
the multi-objective RL agent that controls the local-branching radius. 
Section~\ref{sec:experiments} reports the computational results, and 
Section~\ref{sec:conclusion} concludes with directions for future research.

\section{Literature Review}
\label{litreview}

The reoptimization of vehicle routing problems (VRPs) has received 
significant attention in recent decades, particularly in dynamic contexts 
where uncertainties affect road networks, customer demands, or fleet 
availability. The comprehensive review by~\cite{pillac2013} synthesises 
the dynamic VRP literature and identifies the principal sources of 
uncertainty that motivate online and reactive solution approaches. Within 
this body of work, early contributions emphasised the need for systematic 
frameworks to handle stochastic and unforeseen events. \cite{secomandi2009} 
studied reoptimization policies for the VRP with stochastic demands, 
highlighting the importance of timely policy updates to mitigate 
disruptions. More recently, \cite{florio2022} addressed VRPs with 
stochastic demands and proposed partial reoptimization strategies that 
reduce computational effort while preserving solution quality. 
\cite{niknass2014} developed strategies for reoptimization in a dynamic 
VRP with mixed backhauls, underscoring the role of decision timing in 
maintaining feasible routes. \cite{linfati2018} proposed a heuristic 
reoptimization approach for the capacitated VRP, offering scalable 
solutions suitable for large instances, while \cite{abdallah2017} 
introduced periodic reoptimization mechanisms for dynamic VRPs, 
demonstrating the value of adaptive frameworks in complex and evolving 
networks. Together, these studies form the methodological backbone for 
handling reoptimization in dynamic routing contexts.

A more focused stream of the dynamic VRP literature has explicitly 
addressed disruption recovery, in which routes must be repaired in 
response to discrete events that invalidate the planned schedule. 
\cite{mu2011} formulated the disrupted VRP as a recovery problem and 
proposed two heuristics, tabu search and a Lagrangian relaxation–based 
method, to repair the planned routes after a vehicle break-down or a 
similar incident. \cite{eglese2018} reviewed disruption-related dynamic 
VRP variants and identified the trade-off between solution stability and 
cost as a central but under-explored design dimension across the recovery 
literature. These works frame disruption recovery as a stand-alone 
algorithmic problem rather than a special case of stochastic VRP, and 
motivate the bi-objective treatment adopted in the present paper.

In the forestry sector, the problem of reoptimization has been explicitly 
recognised as critical due to the high exposure to disruptions such as 
delays, equipment failures, and road network changes~\cite{palander2005}. 
Methods have been proposed to dynamically update log-truck routes when 
faced with unforeseen events such as delays or road closures. 
\cite{epstein2007} introduced a rolling horizon approach that incorporates 
real-time data streams to periodically reschedule log-truck operations, 
enabling near-optimal adjustments. \cite{andersson2008} presented RuttOpt, 
a decision support system for log-truck routing that integrates 
optimization with practical constraints and reports efficiency gains of 
up to 30\% in real applications. \cite{amrouss2017} presented a framework 
for the real-time management of transportation disruptions in forestry, 
examining a wide range of disruption scenarios including changes in 
demand, equipment breakdowns, and road network variability, and proposing 
optimization models to mitigate their impact. \cite{amrouss2016} focused 
specifically on the management of unforeseen events, showing how 
real-time data can be leveraged to support rapid reoptimization of forest 
transport operations. The review by~\cite{audy2023} provides a 
comprehensive survey of planning methods and decision support systems in 
timber transportation, emphasising the evolution of solution approaches 
and identifying persistent gaps in addressing dynamic disruptions.

A common limitation across all of these contributions is that each 
proposes a single recovery plan in response to a disruption, leaving 
dispatchers with no structured means to explore the trade-off between plan 
stability and cost efficiency. In practice, the severity of a disruption, 
the time remaining in the planning horizon, and local operational 
priorities all influence which recovery option is most appropriate. To the 
best of our knowledge, no existing work in forestry logistics formulates 
the reoptimization problem as a multi-objective problem and presents 
dispatchers with a Pareto front of non-dominated recovery alternatives. 
This gap directly motivates the design of our framework.

A second methodological stream relevant to this work is the use of local 
branching for multi-objective integer programming. \cite{soylu2015} 
proposed a multi-objective variable neighbourhood search built on the 
local branching idea, collecting Pareto-front segments during the search 
and combining them at termination. \cite{stidsen2014} extended local 
branching as a primal heuristic embedded inside a bi-objective 
branch-and-cut algorithm for binary multi-objective integer programs. 
These works establish that the local branching neighbourhood structure 
can be exploited in multi-objective settings, but the radius is either 
fixed or controlled by classical metaheuristic rules. The contribution of 
the present paper is orthogonal to these works: rather than proposing a 
new multi-objective local branching scheme per se, we use local branching 
as the matheuristic substrate over which a reinforcement learning agent 
adaptively controls the neighbourhood radius for Pareto-front exploration.

The emergence of data-driven approaches has introduced new perspectives 
for adaptive reoptimization. The survey by~\cite{bengio2021} provides a 
methodological tour of machine learning for combinatorial optimization, 
identifying the control of mathematical programming heuristics through 
learned policies as one of the most promising research directions. 
Concrete instantiations of this paradigm have grown rapidly. 
\cite{khalil2016} learned branching variable selection rules for 
mixed-integer programs through imitation learning, while \cite{tang2020} 
employed reinforcement learning to control cut-selection in cutting-plane 
algorithms. In the dynamic routing context, \cite{li2018} 
and~\cite{soeffker2022} have analysed how RL captures complex trade-offs 
across multiple objectives in repeated interaction with the environment. 
Most directly related to our approach, \cite{liu2022learning} demonstrated 
that a reinforcement learning agent can adaptively control the 
neighbourhood search radius in local branching for MILP problems, 
achieving significant reductions in computational effort while maintaining 
solution quality. Our framework builds on the architectural blueprint 
of~\cite{liu2022learning}, but extends it to a multi-objective setting in 
which the agent navigates the deviation--cost Pareto space rather than 
pursuing a single incumbent improvement.

In the context of multi-objective optimization, a growing body of work has 
studied the use of RL to approximate Pareto fronts, an area known as 
multi-objective reinforcement learning (MORL). \cite{roijers2013} provided 
a comprehensive survey of multi-objective sequential decision-making 
frameworks, establishing the theoretical foundations of the field. 
\cite{vanmoffaert2014} developed foundational multi-objective Q-learning 
methods capable of learning sets of Pareto-optimal policies, and showed 
that hypervolume-based reward signals effectively guide exploration across 
the objective space. \cite{xu2020} proposed prediction-guided MORL for 
continuous control tasks, using a learned model to anticipate trade-offs 
across objectives and concentrate exploration on productive regions of the 
Pareto front. \cite{basaklar2023} introduced preference-driven MORL 
algorithms that incorporate explicit preference models into the policy 
network, enabling agents to recover targeted portions of the Pareto front 
according to dispatcher priorities. At the algorithmic level, 
\cite{ropke2024} proposed a divide-and-conquer approach that provably 
converges to the full Pareto front by decomposing the multi-objective 
problem into a sequence of single-objective subproblems, each solved by a 
dedicated oracle, an approach that shares structural similarities with our 
local branching decomposition.

Despite these advances, the application of MORL to combinatorial 
optimization problems in transportation logistics remains largely 
unexplored. Existing MORL methods predominantly operate in continuous or 
small discrete state spaces and assume that each policy evaluation is 
computationally inexpensive. In our setting, by contrast, each evaluation 
of a candidate value of $k$ requires solving a large-scale MILP, making 
the budget of solver calls the primary computational resource to be 
managed. The RL agent in our framework must therefore learn not only which 
regions of the trade-off space are productive, but also when to terminate 
exploration, a problem structure that is absent from standard MORL 
benchmarks. The integration of mathematical programming, local branching, 
and reinforcement learning for forestry logistics remains largely unexplored, 
and to our knowledge this paper presents the first framework that 
addresses the stability--cost trade-off in $\mathcal{LTRSP}$ 
reoptimization through a Pareto-front approach guided by reinforcement 
learning.
\section{Problem Description}
\label{probdesc}
We dedicate this section to describing the $\mathcal{LTRSP}$ reoptimization 
task in response to unforeseen events in forestry operations. We begin by 
highlighting the key characteristics that significantly impact the routing 
and scheduling process. We then introduce the mathematical model of 
$\mathcal{LTRSP}$, which serves as the foundation for tactical planning, and 
finally describe the common unforeseen events that necessitate reoptimization 
in forestry operations.

\subsection{Mathematical Model of the $\mathcal{LTRSP}$}

The $\mathcal{LTRSP}$ represents the tactical planning layer of forestry
logistics. As introduced in~\cite{abdellaoui2025decomposition}, its objective
is to determine a weekly transportation plan that assigns a fleet of trucks
to feasible sequences of trips between forest harvesting sites and processing
mills. Each route consists of alternating empty and loaded trips, together
with loading and unloading operations that must be scheduled within site- and
mill-specific time windows. Compared to classical vehicle routing problems,
the $\mathcal{LTRSP}$ integrates several structural characteristics unique to
forestry operations. Multiple wood assortments and species must be
transported, while each mill imposes product-specific demand requirements
that must be satisfied within the planning horizon. Loader resources are
explicitly modeled, as they constitute critical operating equipment at both
forest sites and mills. The sparse road infrastructure and long travel
distances further restrict route feasibility. In addition, trucks typically
perform multiple trips per day, which requires synchronization across
consecutive trips and strict compliance with driver working-time regulations.

Table~\ref{tab1} presents the sets and parameters that define the model. The
formulation is stated on a space--time network in which each node
$n = (\ell, i)$ couples a physical location $\ell \in \mathcal{L}$ with a time
interval $i \in I$, so that an arc $a \in A$ simultaneously encodes a movement
between two locations and its temporal placement within the day. The decision
variables are ;
\[
\begin{aligned}
x_{avt} &= 1 \text{ if arc } a \text{ is traversed by vehicle } v
           \text{ on day } t, \text{ and } 0 \text{ otherwise,} \\
\delta_{mp} &= \text{unmet demand for product } p \text{ at mill } m.
\end{aligned}
\]

The model~(\ref{objfun})--(\ref{const14}) is denoted as $\mathcal{P}_m$ and
seeks to minimize the total system cost through Equation~\ref{objfun}, which
is composed of two elements: the cost of routing trucks along arcs of the
transportation network, and the penalty cost of unmet demand when mill
requirements cannot be fully satisfied. The optimization problem is subject
to a set of constraints that guarantee the logical and operational
feasibility of the solution:

\begin{itemize}
\item \textbf{Fleet constraints:} Equations~\ref{const1}--\ref{const2} ensure
that each truck departs from and returns to its designated home base at most
once per day.
\item \textbf{Flow conservation:} Equation~\ref{const3} maintains route
continuity at every visited space--time node.
\item \textbf{Demand satisfaction:} Equation~\ref{const4} enforces mill
demand for each product, while allowing unmet demand variables when full
satisfaction is infeasible.
\item \textbf{Supply limitations:} Equation~\ref{const5} ensures that the
harvested volume at each forest site does not exceed its available stock.
\item \textbf{Trip limits:} Equation~\ref{const6} restricts the number of
loaded trips performed by each vehicle on each day.
\item \textbf{Loader availability:} Equations~\ref{const7} and~\ref{const8}
restrict simultaneous loading and unloading operations at forest sites and
mills to the number of loaders available in each time interval.
\item \textbf{Time window constraints:} Equations~\ref{const9}
and~\ref{const10} enforce the opening and closing times of the location
reached by each arc, accounting for the service duration performed there.
\end{itemize}

\begin{table}[H]
\centering
\caption{Model sets and parameters}
\label{tab1}
\renewcommand{\arraystretch}{1.3}
\begin{tabular}{ll}
\toprule
\textbf{Symbol} & \textbf{Definition} \\
\midrule
$\mathcal{L}$ & set of locations representing mills, forest sites, and home bases \\
$F$ & set of forest sites \\
$M$ & set of mills \\
$V$ & set of trucks \\
$P$ & set of wood products \\
$I$ & set of time intervals \\
$N$ & set of space--time nodes $n=(\ell,i)$, with $\ell \in \mathcal{L}$ and $i \in I$ \\
$\mathcal{HB}_v$ & set of eligible home bases for truck $v$ \\
$T$ & set of days in the planning horizon \\
$A$ & set of arcs connecting pairs of nodes across time intervals \\
$A^{+}(n)$ & set of arcs leaving node $n$ \\
$A^{-}(n)$ & set of arcs entering node $n$ \\
$A_{f}^{mp}$ & set of loaded arcs from forest site $f$ to mill $m$ for product $p$ \\
$A_{fi}^{L}$ & loading arcs at forest site $f$ at time interval $i$ \\
$A_{mi}^{U}$ & unloading arcs at mill $m$ at time interval $i$ \\
$c_{a}^{h}$ & hauling cost for arc $a$ \\
$c_{a}^{w}$ & waiting cost for arc $a$ \\
$c_{mp}$ & penalty cost per unit of unmet demand for product $p$ at mill $m$ \\
$i_{a}^{+}$ & start time of the time interval associated with the head node of arc $a$ \\
$d_{mp}$ & weekly demand for product $p$ at mill $m$, expressed in GMT \\
$q_{v}$ & payload capacity of vehicle $v$, expressed in GMT \\
$K_{v}$ & maximum number of trips allowed for vehicle $v$ \\
$\beta_{v}$ & indicator equal to 1 if vehicle $v$ lacks self-loading capability \\
$s_{fp}$ & available supply of product $p$ at forest site $f$, expressed in GMT \\
$S_{a}$ & service duration on arc $a$, corresponding to loading time at forest sites and unloading time at mills \\
$\Lambda_{mi}$ & loader capacity at mill $m$ during time interval $i$ \\
$\Lambda_{fi}$ & loader capacity at forest site $f$ during time interval $i$ \\
$\tau^{o}_{a},\ \tau^{c}_{a}$ & operating window (opening and closing times) at the head location of arc $a$ \\
\textit{BigM} & sufficiently large constant \\
\bottomrule
\end{tabular}
\end{table}

The mathematical model is given as follows:

{\small
\begin{align}
\min \quad
& \sum_{t \in T}\sum_{v \in V}\sum_{a \in A} (c_{a}^{h} + c_{a}^{w})\, x_{avt}
+ \sum_{m \in M}\sum_{p \in P} c_{mp}\, \delta_{mp}
\label{objfun} \\[5pt]
\text{s.t.} \quad
& \sum_{a \in A^{+}(h)} x_{avt} \leq 1,
\quad \forall\, v \in V,\, h \in \mathcal{HB}_v,\, t \in T
\label{const1} \\[5pt]
& \sum_{a \in A^{-}(h)} x_{avt} = \sum_{a \in A^{+}(h)} x_{avt},
\quad \forall\, v \in V,\, h \in \mathcal{HB}_v,\, t \in T
\label{const2} \\[5pt]
& \sum_{a \in A^{+}(n)} x_{avt} = \sum_{a \in A^{-}(n)} x_{avt},
\quad \forall\, v \in V,\, n \in N,\, t \in T
\label{const3} \\[5pt]
& \sum_{t \in T}\sum_{f \in F}\sum_{v \in V}\sum_{a \in A_{f}^{mp}}
q_v\, x_{avt} + \delta_{mp} = d_{mp},
\quad \forall\, m \in M,\, p \in P
\label{const4} \\[5pt]
& \sum_{t \in T}\sum_{m \in M}\sum_{v \in V}\sum_{a \in A_{f}^{mp}}
q_v\, x_{avt} \leq s_{fp},
\quad \forall\, f \in F,\, p \in P
\label{const5} \\[5pt]
& \sum_{m \in M}\sum_{f \in F}\sum_{p \in P}
\sum_{a \in A_{f}^{mp}} x_{avt} \leq K_v,
\quad \forall\, v \in V,\, t \in T
\label{const6} \\[5pt]
& \sum_{v \in V}\sum_{a \in A_{fi}^{L}}
\beta_v\, x_{avt} \leq \Lambda_{fi},
\quad \forall\, i \in I,\, f \in F,\, t \in T
\label{const7} \\[5pt]
& \sum_{v \in V}\sum_{a \in A_{mi}^{U}}
\beta_v\, x_{avt} \leq \Lambda_{mi},
\quad \forall\, i \in I,\, m \in M,\, t \in T
\label{const8} \\[5pt]
& i_{a}^{+} - \textit{BigM}\,(1 - x_{avt}) + S_{a} \leq \tau^{c}_{a},
\quad \forall\, v \in V,\, a \in A,\, t \in T
\label{const9} \\[5pt]
& \tau^{o}_{a} \leq i_{a}^{+} + \textit{BigM}\,(1 - x_{avt}),
\quad \forall\, v \in V,\, a \in A,\, t \in T
\label{const10} \\[5pt]
& x_{avt} \in \{0, 1\},
\quad \forall\, v \in V,\, a \in A,\, t \in T
\label{const11} \\[5pt]
& \delta_{mp} \in \{0, \ldots, d_{mp}\},
\quad \forall\, m \in M,\, p \in P
\label{const14}
\end{align}
}

This tactical model generates a set of truck routes that specify both loaded and empty trips for each day, together with their timing and allocation of resources. The resulting output defines a transportation plan that ensures mill demand is satisfied, forest supply is not exceeded, and operational restrictions are respected.  

A transportation plan consists of a set of routes assigned to different trucks. Each route is composed of a sequence of empty and loaded trips. A truck typically starts from its home base, travels to a forest block to load timber, and then proceeds to a mill to unload the requested products, before continuing with subsequent trips as required. Figure~\ref{fig:truck-route} illustrates an example of a daily truck route.  

Importantly, the transportation plan specifies not only the sequence of visits for each truck but also the expected arrival time at every stop. Figure~\ref{time-route2} provides an example of a truck route where both the sequence of trips and the corresponding arrival times at each location are explicitly represented. 

However, this model assumes deterministic parameters and a disruption-free environment. In reality, unforeseen events frequently occur, making it necessary to embed this tactical model within a broader framework that allows for real-time reoptimization when deviations from the plan arise. Subsequent section presents unforeseen events in forestry.

\begin{figure}[H]
\centering
\begin{tikzpicture}[>=stealth, thick]

    \node (b1) at (0,0) {\includegraphics[width=1.2cm]{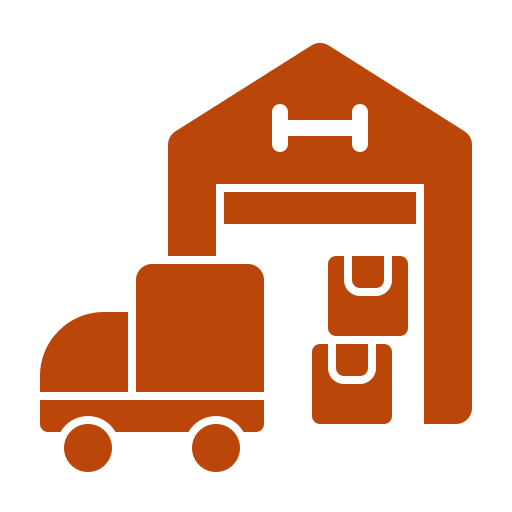}};
    \node (f1) at (3,1.8) {\includegraphics[width=1.2cm]{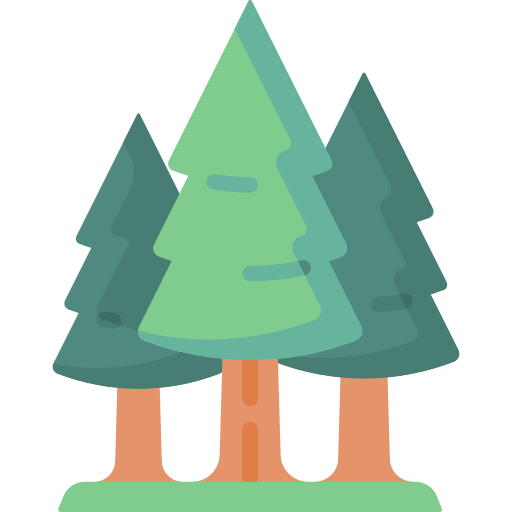}};
    \node (p1) at (6,1.8) {\includegraphics[width=1.2cm]{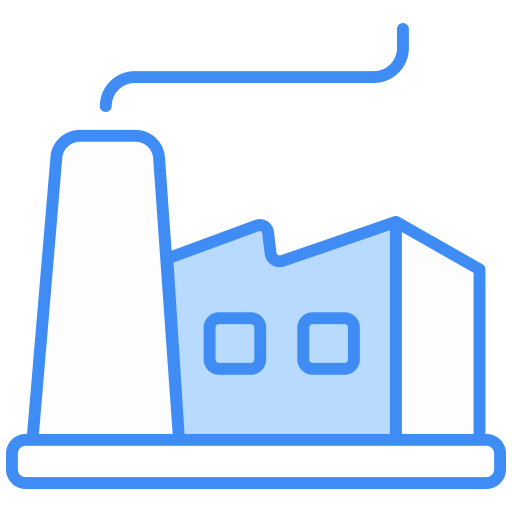}};
    \node (f2) at (3,-0.8) {\includegraphics[width=1.2cm]{trees.png}};
    \node (f3) at (6,-0.8) {\includegraphics[width=1.2cm]{trees.png}};
    \node (p2) at (3,-3) {\includegraphics[width=1.2cm]{power.png}};

    \node[above=0.1cm of f1] {$F_1$};
    \node[right=0.1cm of f2] {$F_2$};
    \node[right=0.1cm of f3] {$F_3$};

    \node[above=0.1cm of p1] {$M_1$};
    \node[below=0.1cm of p2] {$M_2$};
     \node[below=0.05cm of b1] {$HB_1$};
    \draw[->] (b1) -- node[above left] {\scriptsize 1} (f1);
    \draw[->] (f1) -- node[above] {\scriptsize 2} (p1);
    \draw[->, bend right=20] (p1) -- node[right] {\scriptsize 3} (f2);
    \draw[->, bend left=25] (f2) -- node[left] {\scriptsize 4} (p1);
    \draw[->] (p1) -- node[right] {\scriptsize 5} (f3);
    \draw[->] (f3) -- node[below right] {\scriptsize 6} (p2);
    \draw[->] (p2) -- node[below left] {\scriptsize 7} (b1);

\end{tikzpicture}
\caption{Daily truck route : the truck departs from its homebase ($HB1$), visits forest sites ($F_1,F_2,F_3$), delivers raw materials to mills ($M_1,M_2$), and finally returns to homebase. This route is represented by the sequence of trips : $HB_1 \longrightarrow F_1 \longrightarrow M_1 \longrightarrow F_2 \longrightarrow M_1 \longrightarrow F_3 \longrightarrow M_2 \longrightarrow HB_1.$}
\label{fig:truck-route}
\end{figure}

\begin{figure}[H]

\centering
\begin{tikzpicture}[>=stealth, thick, scale=1.0]

\def\step{0.5}
\def\nsteps{22} 

\draw[line width=8pt, orange!70] (0,0) -- (\nsteps*\step,0);
\node[left] at (0,0) { $HB_1$};
\draw[red, thick, fill=red!50] (2*\step,-0.25) rectangle (2.2*\step,0.25);
\draw[red, thick, fill=red!50] (20.5*\step,-0.25) rectangle (20.7*\step,0.25);

\draw[line width=8pt, green!70] (0,-2) -- (\nsteps*\step,-2);
\node[left] at (0,-2) {$F_2$};
\foreach \i in {0,...,\nsteps} {
  \draw (\i*\step,-1.75) -- (\i*\step,-2.25);
}
\draw[red, thick, fill=red!50] (4*\step,-2.25) rectangle (5*\step,-1.75);
\node[above] at (4.5*\step,-2.25) {\scriptsize $i_5$};

\node[above, blue, font=\scriptsize] at (20*\step,-1.75) {};

\draw[line width=8pt, violet!70] (0,-4) -- (\nsteps*\step,-4);
\node[left] at (0,-4) { $M_2$};
\foreach \i in {0,...,\nsteps} {
  \draw (\i*\step,-3.75) -- (\i*\step,-4.25);
}
\draw[red, thick, fill=red!50] (6*\step,-4.25) rectangle (7*\step,-3.75);
\node[above] at (6.5*\step,-4.25) {\scriptsize $i_{7}$};

\draw[red, thick, fill=red!50] (11*\step,-4.25) rectangle (12*\step,-3.75);
\node[above] at (11.5*\step,-4.25) {\scriptsize $I_{12}$};

\draw[red, thick, fill=red!50] (16*\step,-4.25) rectangle (17*\step,-3.75);
\node[above] at (16.5*\step,-4.25) {\scriptsize $i_{16}$};

\draw[line width=8pt, green!50!black!50] (0,-6) -- (\nsteps*\step,-6);
\node[left] at (0,-6) { $F_3$};
\foreach \i in {0,...,\nsteps} {
  \draw (\i*\step,-5.75) -- (\i*\step,-6.25);
}
\draw[red, thick, fill=red!50] (8*\step,-6.25) rectangle (9*\step,-5.75);
\node[above] at (8.5*\step,-6.25) {\scriptsize $i_{9}$};

\draw[red, thick, fill=red!50] (13*\step,-6.25) rectangle (14*\step,-5.75);
\node[above] at (13.5*\step,-6.25) {\scriptsize $i_{14}$};

\draw[->] (0,-8) -- (\nsteps*\step + 0.3,-8) node[right]{\textbf{Time}};
\foreach \i/\lbl in {
  0/4{:}00,  1/4{:}45,  2/5{:}30,  3/6{:}15,
  4/7{:}00,  5/7{:}45,  6/8{:}30,  7/9{:}15,
  8/10{:}00, 9/10{:}45, 10/11{:}30, 11/12{:}15,
  12/13{:}00, 13/13{:}45, 14/14{:}30, 15/15{:}15,
  16/16{:}00, 17/16{:}45, 18/17{:}30, 19/18{:}15,
  20/19{:}00, 21/19{:}45, 22/20{:}30} {
  \draw (\i*\step,-7.9) -- (\i*\step,-8.1)
    node[below, font=\scriptsize, rotate=45, anchor=north east]{\lbl};
}


\draw[->, thick, blue]
  (2.2*\step,0) .. controls +(0,-1) and +(0,1) .. (4*\step,-2)
  node[midway,left,blue]{\scriptsize $T_{b_1,f_2}$};
\node at (4 .5*\step,-0.9) {\includegraphics[width=1cm]{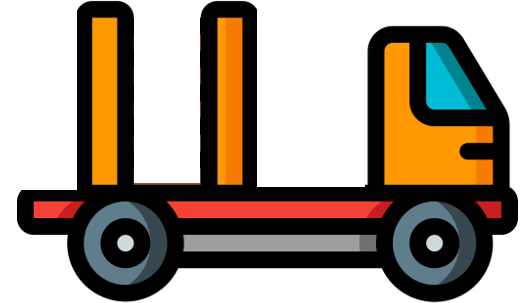}};

\draw[->, thick, blue]
  (5*\step,-2) .. controls +(0,-1) and +(0,1) .. (6*\step,-4)
  node[midway,left,blue]{\scriptsize $T_{f_2,m_2}$};
\node at (7.0*\step,-2.9) {\includegraphics[width=1cm]{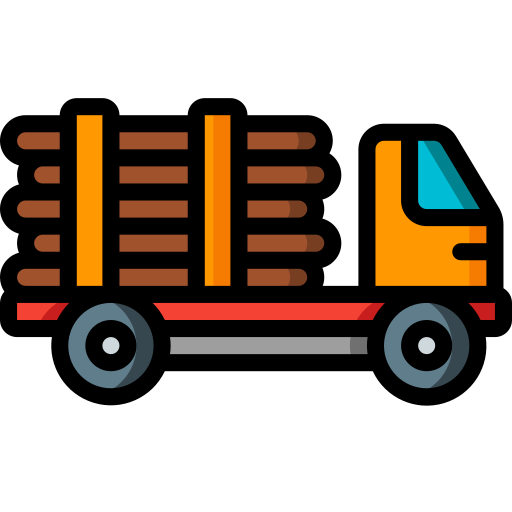}};

\draw[->, thick, blue]
  (7*\step,-4) .. controls +(0,-1) and +(0,1) .. (8*\step,-6)
  node[midway,left,blue]{\scriptsize $T_{m_2,f_3}$};
\node at (4.5*\step,-4.9) {\includegraphics[width=1cm]{truck_empty.png}};

\draw[->, thick, blue]
  (9*\step,-6) .. controls +(0,1) and +(0,-1) .. (11*\step,-4)
  node[midway,left,blue]{\scriptsize $T_{f_3,m_2}$};
\node at (10.6*\step,-5.4) {\includegraphics[width=1cm]{truck_full.png}};

\draw[->, thick, blue]
  (12*\step,-4) .. controls +(0,-1) and +(0,1) .. (13*\step,-6)
  node[midway,right,left]{\scriptsize $T_{m_2,f_3}$};
\node at (13.8*\step,-4.9) {\includegraphics[width=1cm]{truck_empty.png}};

\draw[->, thick, blue]
  (14*\step,-6) .. controls +(0,1) and +(0,-1) .. (16*\step,-4)
  node[midway,right,blue]{\scriptsize $T_{f_3,m_2}$};
\node at (18*\step,-4.9) {\includegraphics[width=1cm]{truck_full.png}};

\draw[->, thick, blue]
  (17*\step,-4) .. controls +(0,1) and +(0,-1) .. (20.5*\step,0)
  node[midway,right,blue]{\scriptsize $T_{m_2,b_1}$};
\node at (19.5*\step,-2.9) {\includegraphics[width=1cm]{truck_empty.png}};

\end{tikzpicture}
\caption{Time-space representation of a daily truck route in the
$\mathcal{LTRSP}$. The truck departs from home base $HB_1$, visits
forest blocks $F_2$ and $F_3$, delivers timber to mill $M_2$ across
three trips, and returns to $HB_1$. Red blocks indicate loading and
unloading operations; arrows represent travel segments.}
\label{time-route2}
\end{figure}

\subsection{Unforeseen Events in Forestry Logistics}

The Canadian forestry sector faces numerous operational challenges arising
from the variability of field conditions and the complexity of managing
transportation across vast geographic areas. Unforeseen events such as road
closures, vehicle breakdowns, extreme weather, and sudden demand changes can
significantly disrupt pre-established transportation plans, making it essential
to design mechanisms for rapid and effective recovery.

Unforeseen events in forestry transportation fall into three broad categories:
\begin{itemize}
    \item \textbf{Network-related disruptions} alter the feasible arc set
    $A$ through road closures caused
    by weather events, accidents, or maintenance work. Unlike urban logistics,
    forestry operations frequently rely on single-access roads, which severely
    limits rerouting flexibility.

    \item \textbf{Supply chain disruptions} modify demand $d_{mp}$ at mills
    $m \in M$ or supply $s_{fp}$ at forest sites $f \in F$, and tend to
    propagate rapidly across multiple trips, affecting the feasibility of
    pre-planned delivery sequences.

    \item \textbf{Operational disruptions} reduce the availability of
    resources such as trucks $v \in \mathcal{V}$ or loader capacities
    $\Lambda_i$ at sites $i \in F \cup M$, often creating bottlenecks
    and leaving equipment idle.
\end{itemize}

Let $\mathcal{E}$ denote the set of all possible unforeseen events. An exact
characterisation of each event $e \in \mathcal{E}$ is the natural starting
point for updating the routing and scheduling plan. Each event is described
by two complementary components. Its \textit{intrinsic attributes} identify
what happened: the event type $t_e$ (road closure, breakdown, demand
variation, etc.), its location $l_e$ in the network, its occurrence time
$\tau_e$, and its duration $d_e \in \mathbb{R}_+$. Its \textit{induced
impact} captures how it modifies the model: the set $\Delta A_e$ encodes
arc removals or additions in the routing network, while $\Delta \theta_e$
encodes changes to problem parameters such as supply $s_{fp}$, demand
$d_{mp}$, or resource capacities $\Lambda_i$. Formally, each event
$e \in \mathcal{E}$ is represented as the tuple
\begin{equation}
    e \;\mapsto\;
    \bigl(\,t_e,\; l_e,\; \tau_e,\; d_e,\; \Delta A_e,\;
    \Delta \theta_e\,\bigr)
    \label{eq:event_tuple}
\end{equation}
which provides the necessary information to transform the feasible region
$\mathcal{F}$ of the optimization model $\mathcal{LTRSP}$ into its disrupted counterpart
$\mathcal{F}(e)$, thereby enabling 
reoptimization under the new operational
conditions.
Table~\ref{tab:disruptions} details the selected subset of events from the three categories above, chosen because of their high frequency in practice, as reported by our industrial partner.

\begin{sidewaystable}
\centering
\caption{Detailed characterization of unforeseen events: description, required information, and technical impact on the model.}
\begin{tabular}{|p{3cm}|p{5cm}|p{5cm}|p{6cm}|}
\hline
\textbf{Event type} & \textbf{Description} & \textbf{Needed information} & \textbf{Technical impact (on network and data)} \\
\hline
 
\textbf{Road / site closures} 
& Temporary shutdown of forest sites, mills, or access roads. Since most forest roads are single-access, their closure severely limits rerouting possibilities. 
& Occurrence time $\tau$, exact location of the closure, duration $d_e$, type of site affected (mill/forest), and list of trips completed prior to the event. 
& Eliminate all incoming and outgoing arcs of the closed node or road during $t \in [\tau,\tau+d_e]$, i.e., $A \gets A \setminus A_e$. Waiting arcs are preserved. This reduces feasible routing options and may render the model infeasible. \\
\hline

\textbf{Delays} 
& Travel slowdowns caused by adverse weather, road deterioration, traffic congestion, seasonal hunting/fishing restrictions, or vehicle malfunctions.  
& Occurrence time $\tau$, delay duration $d_e$, affected arcs $(i,j)\in A$, truck identifier if specific, and remaining distance.  
& Update arc travel times: $\tau_{ij} \gets \tau_{ij} + \Delta \tau_{ij}$. Adjust duration constraints and propagate shifts in the time--space network. For vehicle-specific delays, reconnect its start vertex at the delayed time. If $d_e$ exceeds the planning horizon, remove the truck’s arcs. \\
\hline

\textbf{Truck breakdowns} 
& A vehicle is out of service during repair. Equivalent to a delay if the repair time $d_e$ is within the horizon.  
& Occurrence time $\tau$, identifier of the affected truck $v \in \mathcal{V}$, repair duration $d_e$, and its current location.  
& Remove truck $v$ from the available fleet: $\mathcal{V} \gets \mathcal{V}\setminus\{v\}$ for $t \in [\tau,\tau+d_e]$. If the repair ends within the horizon, reconnect its source vertex at the repair completion time. \\
\hline

\textbf{Demand / supply variations} 
& Changes in mill demand or forest supply, e.g., mill capacity reduction due to breakdowns or updated demand forecasts.  
& Occurrence time $\tau$, location of the disruption (mill $m$ or forest $f$), product $p$, variation $\Delta d_{mp}$ or $\Delta s_{fp}$, and executed trips.  
& Modify parameters $d_{mp}$ and/or $s_{fp}$. Update balance constraints: 
\[
\sum_{a \in \delta^-(m)} x_{ap} - \sum_{a \in \delta^+(m)} x_{ap} = d_{mp}, \quad \forall m,p.
\]  
Enable or remove arcs between mills and sites if new flows become feasible. \\
\hline

\textbf{Loader breakdowns} 
& Reduced or unavailable loading/unloading capacity at mills or forest sites. A full breakdown is equivalent to a site closure.  
& Occurrence time $\tau$, affected site $i \in F \cup M$, repair duration $d_e$, number of loaders before and after breakdown.  
& Decrease loader capacity: $\Lambda_i \gets \Lambda_i - \Delta \Lambda_i$. Update service constraints: 
\[
\sum_{v \in \mathcal{V}} \sum_{a \in \delta^-(i)} x_{avt} \leq \Lambda_i, \quad \forall i,t.
\]  
If $\Lambda_i = 0$, remove all arcs entering or leaving $i$ for $t \in [\tau,\tau+d_e]$. \\
\hline

\end{tabular}
\label{tab:disruptions}
\end{sidewaystable}
When an event $e \in \mathcal{E}$ occurs at time $\tau_e$, the current transportation plan may be affected through delays, trip cancellations, or parameter variations. At the occurrence, we assume that real-time information becomes available regarding the nature of the disruption and the current state of operations. This includes the location and status of every truck (loaded/empty, at mill, forest site, or in transit), the set of completed trips, and the set of remaining trips. Such information is used to update the time–space network and reoptimize the transportation plan in response to $e$. Formally, the updated model $\mathcal{P}_m(e)$ fixes decision variables associated with already executed trips, removes infeasible arcs, and reschedules the remaining activities under the new feasibility domain $\mathcal{F}(e)$.  
\section{Methodology}
\label{sec:methodology}
In this section, we present a unified methodology for the reoptimization of log-truck routes and schedules in response to unforeseen events. 
The overall framework, illustrated in Figure~\ref{reoptimization-arche}, integrates five interrelated components that operate in a continuous feedback loop to ensure adaptive and resilient decision-making. The proposed framework, illustrated in Figure~\ref{reoptimization-arche}, is 
organized around a dynamic memory that decouples two operational layers. The 
upstream layer maintains the current plan and operational state and runs 
continuously; the downstream layer activates only when a disruption is detected, 
reads the latest plan and system state from the memory, and writes back a 
recovered plan. This separation lets the framework absorb successive events 
asynchronously without interrupting nominal monitoring.

The upstream layer comprises the Optimization Engine, which produces the initial 
tactical plan $S_0$ following \cite{abdellaoui2025decomposition}, and Real-Time 
Monitoring, which tracks plan execution and emits the system state $S_t$ from 
standard fleet-management infrastructure. Neither is a contribution of this paper.

For the Unforeseen Events component shown in Figure~\ref{reoptimization-arche}, 
in a deployment setting disruptions would be communicated through monitoring 
systems or human interactions such as calls and emails from drivers, dispatchers, 
or mill operators. In this study, we instead rely on an Unforeseen Events Generator 
that supplies the disruption scenarios used to evaluate the framework. We start 
from a list of real unforeseen events collected through intensive discussions with 
our forestry partner and enrich it with synthetically generated events that share 
the same characteristics (type, frequency, duration, location) to produce a 
broader set of test scenarios.

The Disruption Characterization and Impact Module  maps 
each event $e$ to a structured pair $\varphi(e) = \big(\Delta\theta(e), 
\Delta\mathcal{A}(e)\big)$ of parameter and arc-set updates. The Reoptimization 
Engine then computes a recovered plan $S_{t+1}$ 
via a local-branching MILP whose neighbourhood radius is adaptively controlled by 
a reinforcement learning agent, returning a Pareto front of stability--cost 
trade-offs to the dispatcher. These three components, namely the Unforeseen Events 
Generator, the DCIM, and the Reoptimization Engine, constitute the methodological 
core of this paper and are detailed in the subsequent subsections.

\begin{figure}[H]
\centering
\includegraphics[width=0.9\textwidth]{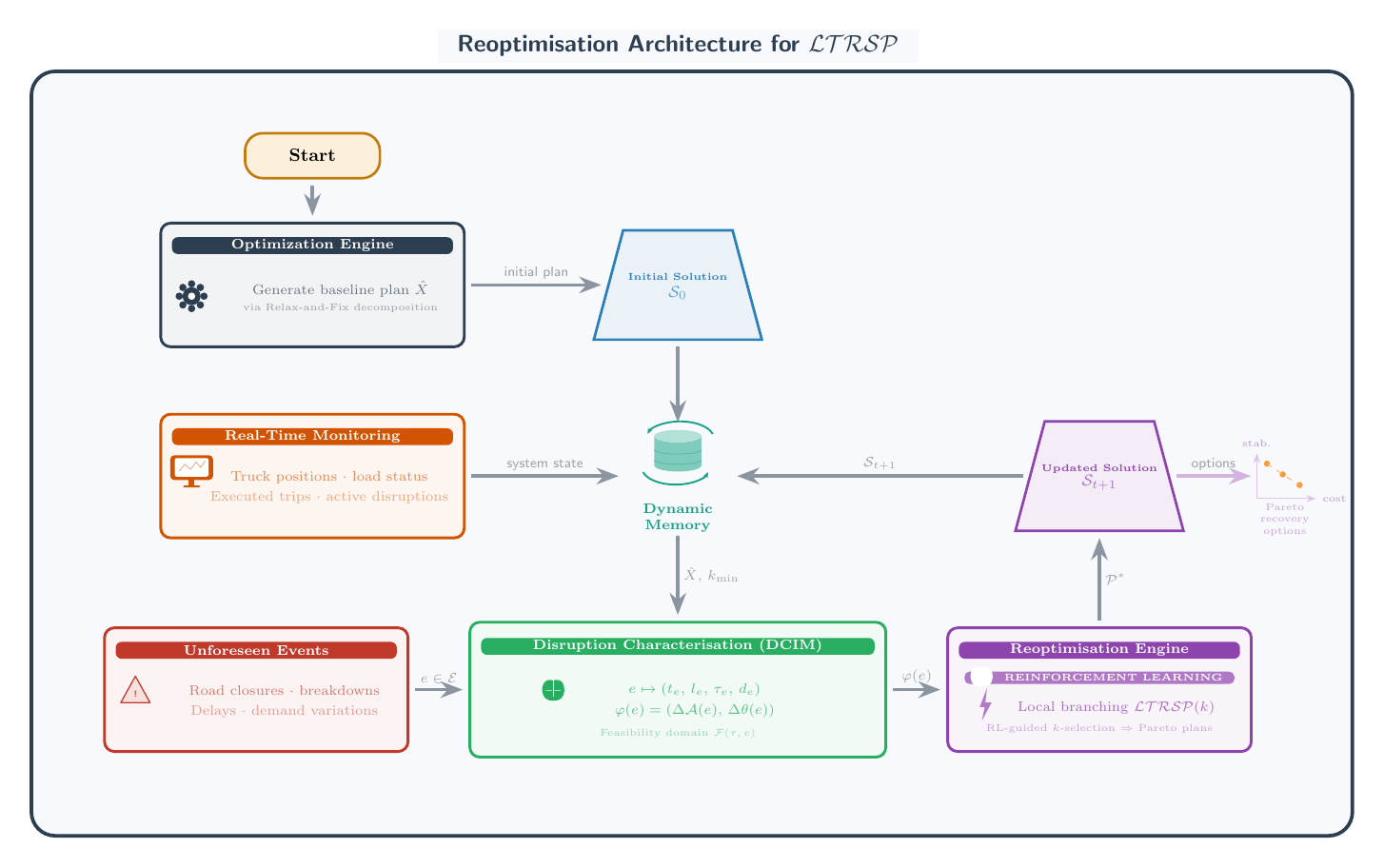}
\caption{Continuous reoptimization architecture for the $\mathcal{LTRSP}$. The initial plan $\mathcal{S}_0$ is stored in the dynamic memory, which is continuously updated by real-time. When unforeseen events appear, they are detected and characterized by the DCIM, which uses the current system state and event information to update the model. The updated model is reoptimized to produce a new plan $\mathcal{S}_{t+1}$, which is then stored back in the dynamic memory, ensuring continuous adaptation.}
\label{reoptimization-arche}
\end{figure}

The main contribution of this paper lies in the design and integration of the  third, fourth and fifth components of the proposed framework, 
namely  \emph{Unforeseen Events Generator}, \emph{Disruption Characterization and Impact Module (DCIM)}, and the \emph{Reoptimization Engine}. 
As previously mentioned, the first component follows the deterministic formulation described in \cite{abdellaoui2025decomposition}, 
while the second component, generally relies on existing monitoring infrastructures such as sensors, 
GPS tracking systems, and fleet management software to collect and transmit real-time operational data. 

The subsequent subsections therefore focus on providing a detailed description of these components, 
as well as their mathematical integration within the overall reoptimization architecture. 

\subsection{Unforeseen Events Generator}
\label{sec:event-generator}

The first methodological component is an Unforeseen Events Generator, 
responsible for producing disruption scenarios to be applied to problem 
instances. This module is inspired by stochastic simulation approaches in 
the resilience-testing literature, tailored to the context of the 
$\mathcal{LTRSP}$ in collaboration with our forestry partner.

Formally, let $\mathcal{E}$ denote the set of possible unforeseen events, 
and let $\Omega$ be the set of scenarios. Each scenario $\omega \in \Omega$ 
corresponds to a subset $\mathcal{E}_\omega \subseteq \mathcal{E}$ of events 
occurring at specific times and locations during the planning horizon. The 
generator stochastically draws $\mathcal{E}_\omega$ from predefined 
distributions over event type, frequency, location, and duration.

We do not claim that the distributional parameters in 
Table~\ref{tab:event_distributions} are statistically estimated from 
historical incident records. Rather, they constitute a parameterised 
stress-test specification informed by extensive discussions with our 
industrial partner regarding the event categories, typical durations, and 
approximate frequencies observed in operations. The purpose of the 
generator is to systematically span the relevant range of disruption 
severities, rather than to reproduce the empirical disruption distribution 
of any specific operation. This framing is consistent with the 
resilience-testing literature, where stochastic generators are used to 
stress-test recovery procedures across structurally diverse scenarios 
rather than to forecast disruption occurrence~\citep{wang2023resilience}.

The output of the generator is a collection of scenarios
\[
\Omega = \{\mathcal{E}_1, \mathcal{E}_2, \dots, \mathcal{E}_{|\Omega|}\}, 
\qquad
\mathcal{E}_\omega = \{e_1, e_2, \ldots, e_{n_\omega}\},
\qquad
e_i = (t_i, \; l_i, \; \tau_i,\; d_i),
\]
where each $\mathcal{E}_\omega$ contains the detailed event specifications 
later translated into model modifications. This construction ensures that 
the reoptimization framework is tested across a broad set of disruption 
configurations rather than on a limited number of ad-hoc cases. The 
synthetic disruptions are generated according to the rules summarised in 
Table~\ref{tab:event_distributions}.
\subsection{Disruption characterization and impact module (DCIM)}

The second methodological block is the \emph{Disruption Characterization and Impact Module (DCIM)}, which takes the raw events produced by the generator and translates them into explicit modifications of the mathematical model. The DCIM has two complementary functions:  

\begin{enumerate}
    \item \textbf{Event characterization:} Each event $e \in \mathcal{E}$ is described by its intrinsic attributes, such as type, location, occurrence time, and duration. Formally,
    \[
    e \mapsto \big(t_e, \; l_e, \tau_e,\; d_e \big).
    \]
    \item \textbf{Impact mapping:} Each event is then mapped to a set of technical impacts, namely (i) modifications to the routing network $A$ and (ii) adjustments to the problem parameters $\theta$ (e.g., demand, supply, resource capacities). This mapping can be expressed as
    \[
    \varphi: e \;\longmapsto\; \big(\Delta A(e), \; \Delta \theta(e)\big),
    \]
    where $\Delta A(e)$ represents arc additions/removals in the time–space network and $\Delta \theta(e)$ represents parameter updates.
\end{enumerate}. The rules of this mapping are event-specific, as detailed in Table~\ref{tab:disruptions}. For example, a road closure event produces $\Delta A(e) = A \setminus A_e$ (removal of arcs), while leaving $\theta$ unchanged. Conversely, a demand variation modifies $d_{mp}$ in $\theta$ but leaves $A$ unchanged. More complex events, such as truck breakdowns, induce simultaneous changes in both sets. The DCIM therefore serves as the formal bridge between disruption scenarios and their integration into the reoptimization model. It ensures that each scenario $\omega$ generated by the first module is consistently and rigorously  operationalized in terms of network feasibility and parameter constraints, thereby enabling systematic reoptimization of the $\mathcal{LTRSP}$.  
\begin{table}[H]
\centering
\caption{Stochastic generation rules for unforeseen events in the $\mathcal{LTRSP}$.}
\resizebox{\textwidth}{!}{
\begin{tabular}{lllll}
\hline
\textbf{Event type} & \textbf{Distribution law} & \textbf{Parameters} & \textbf{Affected element} & \textbf{Model impact} \\ \hline
Road closure & Bernoulli($p=0.08$), duration $U[1,3]$ days & Independent by forest road & Arc $(i,j) \in A$ & $\Delta A(e)$: arc removal for duration \\ 
Truck breakdown & Poisson($\lambda=0.05$/day), repair $U[4,12]$ h & Independent per vehicle $v \in \mathcal{V}$ & Truck $v$ & $\Delta \theta(e)$: downtime and availability shift \\ 
Travel delay & Exponential($\lambda=0.15$) & Arc-dependent mean delay $\mu_{ij}$ & Arc $(i,j)$ & $\Delta \theta(e)$: travel time increase \\ 
Demand variation & Normal($\mu=d_{mp}$, $\sigma=0.1d_{mp}$) & Mill–product pair $(m,p)$ & Demand parameter $d_{mp}$ & $\Delta \theta(e)$: right-hand side change \\ 
\hline
\end{tabular}
}
\label{tab:event_distributions}
\end{table}
\subsection{Reoptimization Engine}
\label{sec:plb-method}

The reoptimization engine is responsible for generating a new routing and
scheduling plan that explicitly accounts for the operational impacts of
unforeseen events, represented by the perturbation sets
$(\Delta\mathcal{A}, \Delta\theta)$. Its primary objective is to minimize
the total routing and scheduling cost while restoring feasibility and
maintaining a high degree of consistency with the original tactical plan.
In operational practice, drastic modifications to the baseline schedule are
undesirable, as they may disrupt the commitments and working patterns of
truck drivers, contractors, and mill operators, who organize their
activities around the initially established plan. Consequently, the
reoptimization process must simultaneously ensure two properties: cost
efficiency by minimizing the overall operating cost, and stability by
preserving as much as possible the structure of the pre-disruption plan.

Through extensive discussions with forest planners and dispatchers from our
industrial partner, a clear operational requirement emerged regarding how
these two objectives should be handled in practice. When a disruption
occurs, planners do not want a single prescribed recovery plan imposed by
the optimization system. Rather, they need a structured set of recovery
options, each representing a different compromise between plan stability
and cost efficiency, so that they can select the most appropriate response
depending on the severity of the disruption, the time remaining in the
planning horizon, and their knowledge of field conditions. A highly stable
plan that closely mirrors the original schedule may be preferred when the
disruption is minor or when coordination costs are high, while a more
aggressive restructuring may be acceptable when the potential cost savings
are substantial. This insight fundamentally shapes the design of our
reoptimization engine: rather than returning a single solution, the
framework constructs a Pareto front of non-dominated recovery plans
spanning the stability-cost space and presents it to the dispatcher as an
explicit menu of options.

Let $e$ be an unforeseen event occurring at time $t_e$. After detection,
its operational impact is translated through the DCIM, generating
perturbations $(\Delta\mathcal{A}, \Delta\theta)$ on the arc set and model
parameters. This induces a new feasibility domain, denoted by
$\mathcal{F}(e)$ for the $\mathcal{LTRSP}$ model $\mathcal{P}_m$. The
reoptimization problem then consists of solving $\mathcal{P}_m$ over the
updated feasibility domain $\mathcal{F}(e)$:
\[
\min \; z(X) \quad \text{s.t. } X \in \mathcal{F}(e).
\]
Moreover, all operational decisions realized before the event occurrence
time $t_e$ are fixed. Let $X^{*}$ denote the original tactical solution,
with components $x^{*}_{atv}$. We impose the fixing constraints:
\[
x_{atv} = x^{*}_{atv}, \quad
\forall (a,t,v) \text{ such that } t \leq t_e.
\]
Thus, reoptimization only affects decisions corresponding to operations
scheduled after $t_e$, while the pre-disruption portion of the plan
remains unchanged. The updated model $\mathcal{P}_m^{\mathrm{reop}}$ can
be solved directly to reoptimize the $\mathcal{LTRSP}$ under event $e$.
However, solving it without additional restrictions may yield a solution
that differs substantially from the original solution $X^{*}$. To ensure
stability, we instead search for a new solution that satisfies all
constraints of $\mathcal{P}_m^{\mathrm{reop}}$ while remaining within a
controlled neighbourhood of $X^{*}$. To this end, a local branching
approach is employed. The next subsection introduces the technical details.

\subsubsection{Local Branching}

To balance stability and solution quality, the reoptimization procedure
solves $\mathcal{P}_m^{\mathrm{reop}}$ under a local branching framework.
Intuitively, when the search is restricted to a narrow neighbourhood
around the original plan, stability is high and computational time is
limited. Expanding the neighborhood allows greater flexibility, which may improve solution quality, but at the cost of increased runtime and reduced stability. 

Accordingly, the proposed approach introduces a neighborhood size parameter that controls this trade-off between stability and solution quality, following the local branching scheme of \cite{fischetti2003local}. The deviation from the pre-disruption tactical plan is incorporated directly into the $\mathcal{P}_m^{reop}$  through a constraint that measures the Hamming distance between the original and reoptimized solutions, as illustrated in Figure~\ref{fig:polyhedron_local_branching}. The resulting $\mathcal{P}_m^{reop}$ with the local branching constraint can be formulated as follows:
\begin{figure}[H]
\centering
\includegraphics[width=0.4\textwidth]{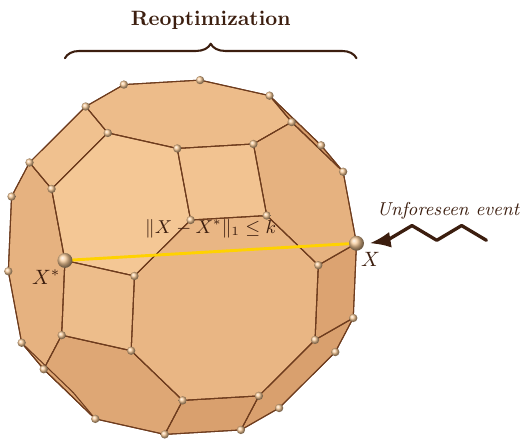}
\caption{\centering Illustration of the feasible polyhedral region $\mathcal{P}_m^{reop}$. The distance between $X^*$ and $X$ represents the local branching deviation bounded by $\|X - X^*\|_1 \le k$.}
\label{fig:polyhedron_local_branching}
\end{figure}
Using binary variables \(x_{atv}\) of $\mathcal{P}_m^{reop}$ indicating whether truck \(v \in \mathcal{V}\) uses arc \(a \in A\) at time \(t \in \mathcal{T}\), 
and \(x_{atv}^{*}\) its value in the pre-disruption tactical plan $X^{*}$. 
The deviation from the baseline configuration is captured by auxiliary binary variables \(y_{atv}\) defined as:
\[
y_{atv} \ge x_{atv} - x_{atv}^{*}, 
\qquad
y_{atv} \ge x_{atv}^{*} - x_{atv},
\]
so that \(y_{atv} = |x_{atv} - x_{atv}^{*}|\) at optimality.  
The reoptimization problem can then be expressed as:
\begin{equation}
\label{eq:PLB}
\min_{X \in \mathcal{F}(e)} 
\sum_{t \in \mathcal{T}}\sum_{v \in \mathcal{V}}\sum_{a \in A} c_a x_{atv}
\end{equation}
\begin{equation}
\label{eq:PLBconstr} \sum_{(a,t,v)} y_{atv} \leq k
\end{equation}
such that  $\mathcal{F}(e)$ is the new feasibilty domain under the new disruptions events and $k$  is the neighborhood search size. 
The structure of $\mathcal{P}_m^{\mathrm{reop}}$ requires the existence of a
minimal deviation value $k_{\min}$ that guarantees model feasibility after a
disruption. Solving $\mathcal{P}_m^{\mathrm{reop}}(k_{\min})$ yields the most
stable recovery plan, obtained quickly thanks to the restricted search space,
although its routing cost may be suboptimal. This solution thus serves as the
natural anchor of the Pareto front. The exact computation of $k_{\min}$ is
addressed in the following subsection.
\subsubsection{Analytical Bounds on the Minimal Feasible Deviation}
\label{eq:kmin_def}

The exploration of the stability--cost front takes place over an interval of
deviation radii $[k_{\min}, k_{\max}]$. Its lower end anchors the front at its
most stable extreme and must be exact, since it is the radius at which the first
recovery plan is produced. Its upper end bounds the exploration and must be
large enough not to exclude relevant plans, yet small enough to keep the search
finite. This section establishes both.

\medskip
\noindent
\textbf{Definition of the anchor.}
The minimal feasible deviation is
\begin{equation}
\label{eq:kmin-def}
k_{\min}
\;=\;
\min_{X \in \mathcal{F}(e)} \bigl\lVert X - X^{*} \bigr\rVert_1 ,
\end{equation}
that is, the smallest number of binary changes required to transform the
baseline tactical plan $X^{*}$ into a plan that is feasible for the disrupted
region $\mathcal{F}(e)$. Since $X, X^{*} \in \{0,1\}^{n}$, the $\ell_1$-distance
coincides with the Hamming distance and $k_{\min} \in \mathbb{Z}_{\geq 0}$. The
value $k_{\min} = 0$ is admissible and operationally meaningful: it indicates
that the baseline plan remains feasible under the disruption, which occurs for
instance after a demand variation that degrades cost without violating any
constraint.

\medskip
\noindent
\textbf{Exact computation.}
We compute $k_{\min}$ exactly, by solving the mixed-integer programme that
minimises the Hamming distance to the baseline over the disrupted feasible
region. Introducing, for every decision $i$, two binary variables
$\delta^{+}_{i}$ and $\delta^{-}_{i}$ that account respectively for an
activation and a deactivation with respect to $X^{*}$, the absolute value is
linearised and the programme reads
\begin{equation}
\label{eq:kmin-milp}
\begin{aligned}
k_{\min} \;=\; \min_{X,\, \delta^{+},\, \delta^{-}} \quad
& \sum_{i=1}^{n} \bigl( \delta^{+}_{i} + \delta^{-}_{i} \bigr) \\
\text{s.t.} \quad
& \delta^{+}_{i} \;\geq\; X_{i} - X^{*}_{i},
  \qquad i = 1,\dots,n, \\
& \delta^{-}_{i} \;\geq\; X^{*}_{i} - X_{i},
  \qquad i = 1,\dots,n, \\
& X \in \mathcal{F}(e), \\
& X \in \{0,1\}^{n}, \quad
  \delta^{+}, \delta^{-} \in \{0,1\}^{n} .
\end{aligned}
\end{equation}
The feasible region $\mathcal{F}(e)$ is expressed by the complete constraint
set of the reoptimization model: flow conservation, demand satisfaction under
the perturbed right-hand side, supply, loader and time-window restrictions,
together with the arc removals $\Delta\mathcal{A}(e)$ induced by the event and
the fixing of every decision completed before $\tau_e$. Formulation
\eqref{eq:kmin-milp} therefore shares exactly the feasible region of
$\mathcal{P}^{\mathrm{reop}}_m(\cdot)$ and differs from it only in the
objective.

\begin{proposition}[Exactness of the anchor]
\label{prop:kmin-exact}
Let $(X, \delta^{+}, \delta^{-})$ be an optimal solution of
\eqref{eq:kmin-milp}. Then
$\delta^{+}_{i} + \delta^{-}_{i} = |X_{i} - X^{*}_{i}|$ for every $i$, and the
optimal value of \eqref{eq:kmin-milp} equals $k_{\min}$ as defined
in~\eqref{eq:kmin-def}. Moreover, the associated plan $X$ attains the minimum
in~\eqref{eq:kmin-def}, so it is a most stable feasible recovery plan.
\end{proposition}

\begin{proof}
Since $X_i, X^{*}_i \in \{0,1\}$, at most one of $X_i - X^{*}_i$ and
$X^{*}_i - X_i$ is positive, and the constraints impose
$\delta^{+}_{i} + \delta^{-}_{i} \geq |X_{i} - X^{*}_{i}|$. As the objective
minimises this sum with non-negative coefficients, equality holds at any
optimum. The objective therefore coincides with $\lVert X - X^{*}\rVert_1$ over
$\mathcal{F}(e)$, and minimising it returns $k_{\min}$ together with a
minimiser.
\end{proof}

Solving \eqref{eq:kmin-milp} to optimality is not a computational obstacle in
our setting, for three reasons. The routing cost is absent from the objective,
which removes the coupling that makes the reoptimization model difficult and
leaves a pure feasibility-restoration problem. The fixing of all decisions
completed before $\tau_e$ eliminates a large share of the binary variables,
the more so as the event occurs later in the week. Finally, the disruption
restrictions themselves prune the arc set. In our experiments,
\eqref{eq:kmin-milp} is solved to proven optimality within seconds, against
minutes for a single evaluation of $\mathcal{P}^{\mathrm{reop}}_m(k)$.

Exactness matters operationally, and is the reason we solve
\eqref{eq:kmin-milp} rather than resorting to a primal heuristic. A radius
below $k_{\min}$ renders $\mathcal{P}^{\mathrm{reop}}_m(k)$ infeasible, so an
underestimate wastes a solver call and returns no plan at all. An overestimate,
conversely, silently discards the most stable recovery plans, precisely those
that dispatchers value most when the disruption is minor. Since the anchor also
calibrates the step sizes of the search, an inexact value would propagate to the
entire exploration.

\medskip
\noindent
\textbf{Structural cap.}
Let $\mathcal{S}^{*} = \{ i : X^{*}_{i} = 1 \}$ denote the set of decisions
active in the baseline plan and $n_{\mathrm{act}} = |\mathcal{S}^{*}|$ its
cardinality. We set
\begin{equation}
\label{eq:kmax-structural}
k_{\max} \;=\; n_{\mathrm{act}} .
\end{equation}
At this radius the recovered plan may differ from the baseline in every
movement the baseline performs. Beyond it, the local branching constraint no
longer conveys any stability requirement: the recovered plan bears no
resemblance to the schedule around which drivers, loaders and mill operators
have organised their day, which is the very property the constraint is meant to
preserve. Radii larger than $n_{\mathrm{act}}$ are therefore excluded on
operational grounds rather than computational ones.

\medskip
\noindent
\textbf{Determination of the effective cap by relaxation.}
The structural cap bounds the search but does not indicate when it becomes
useless to continue. This is obtained from the continuous relaxation. Let
$z^{*}(k)$ denote the optimal cost of $\mathcal{P}^{\mathrm{reop}}_m(k)$ and let
\begin{equation}
\label{eq:lp-floor}
z_{\mathrm{LP}}
\;=\;
\min \Bigl\{\, z(X) \;:\; X \in [0,1]^{n},\;
X \text{ satisfies the constraints of } \mathcal{F}(e),\;
\lVert X - X^{*} \rVert_1 \leq k_{\max} \,\Bigr\}
\end{equation}
be the value of the linear relaxation obtained by replacing $X \in \{0,1\}^n$
with $X \in [0,1]^n$ at the structural cap. This single linear programme is
solved once, at initialization, in a few seconds.

\begin{proposition}[LP floor and effective cap]
\label{prop:lp-floor}
For every $k \in [k_{\min}, k_{\max}]$,
\begin{equation}
\label{eq:floor-ineq}
z_{\mathrm{LP}} \;\leq\; z^{*}(k_{\max}) \;\leq\; z^{*}(k) .
\end{equation}
Consequently, if a radius $\hat{k}$ satisfies
$z^{*}(\hat{k}) \leq z_{\mathrm{LP}} + \varepsilon$, then no radius
$k > \hat{k}$ can improve the cost by more than $\varepsilon$, and $\hat{k}$ is
an effective cap for the exploration.
\end{proposition}

\begin{proof}
The relaxed feasible set at $k_{\max}$ contains the integer feasible set at
$k_{\max}$, hence $z_{\mathrm{LP}} \leq z^{*}(k_{\max})$. The nesting
$\mathcal{X}(k) \subseteq \mathcal{X}(k_{\max})$ for $k \leq k_{\max}$ gives
$z^{*}(k_{\max}) \leq z^{*}(k)$, which yields~\eqref{eq:floor-ineq}. If
$z^{*}(\hat{k}) \leq z_{\mathrm{LP}} + \varepsilon$, then for any
$k > \hat{k}$ we have
$z^{*}(k) \geq z_{\mathrm{LP}} \geq z^{*}(\hat{k}) - \varepsilon$, so the
attainable improvement is at most $\varepsilon$.
\end{proof}

The interval $[k_{\min}, k_{\max}]$ is thus delimited from below by an exact
value and from above by a structural cap refined at run time by the LP floor.
The relative distance to that floor,
\begin{equation}
\label{eq:gap-lp}
g \;=\;
\frac{z^{*}(k) - z_{\mathrm{LP}}}{\max\bigl\{ |z^{*}(k)|,\, 1 \bigr\}} ,
\end{equation}
measures the improvement potential that remains and is supplied to the
reinforcement learning agent of Section~\ref{sec:rl_approach} as a state
feature. A value close to zero certifies that the staircase has reached its
floor and that terminating is rational; a strictly positive value indicates
that a larger radius may still uncover a non-dominated plan, and justifies
crossing a plateau rather than stopping on it. The agent of
Section~\ref{sec:rl_approach} accordingly starts its exploration at $k_{\min}$
and navigates $[k_{\min}, k_{\max}]$ under this certificate.
\subsubsection{Reinforcement learning for adaptive Pareto front exploration}
\label{sec:rl_approach}

Solving the local branching model $\mathcal{P}^{\mathrm{reop}}_m(k)$ at a single
value of $k$ yields a single recovery plan. As mentioned before, dispatchers require instead a set of
alternatives spanning the stability--cost trade-off. This set is the Pareto front
of the bi-objective problem $\min\,(k,\, z^{*}(k))$, in which the best achievable
cost at deviation $k$ and the corresponding feasible set are
\begin{equation}
    z^{*}(k) = \min\{\, z(X) : X \in \mathcal{X}(k) \,\},
    \qquad
    \mathcal{X}(k) = \mathcal{F}(e) \cap
    \bigl\{ X \in \{0,1\}^{n} : \lVert X - X^{*}\rVert_1 \leq k \bigr\}.
    \label{eq:nested_sets}
\end{equation}
Enumerating this front over a dense grid of values is computationally
prohibitive: each evaluation of $z^{*}(k)$ requires solving a large-scale MILP,
and most values yield no new non-dominated plan.
\paragraph{Structure of the front.}
The sets $\mathcal{X}(k)$ are nested, $\mathcal{X}(k) \subseteq \mathcal{X}(k')$
for $k \leq k'$, so $z^{*}(\cdot)$ is non-increasing and piecewise constant: it
is a staircase whose steps are exactly the non-dominated points. Two
consequences are exploited by the search. First, the front is anchored at
$k_{\min}$, the smallest feasible value. Second, detecting a new non-dominated
point reduces to a scalar test on the running envelope
\begin{equation}
    \bar{z}(k) \;=\; \min\{\, \hat{z}(k') \;:\; k' \in \mathcal{K},\ k' \leq k \,\},
    \label{eq:envelope}
\end{equation}
where $\mathcal{K}$ is the set of values already evaluated and $\hat{z}(k')$ the
objective recorded at $k'$: a solve at $k$ is Pareto-improving if and only if it
reduces $\bar{z}(k)$ by more than the absolute optimality tolerance of the
solver.

The agent interacts with the reoptimization environment as a Markov decision
process $(\mathcal{S}, \mathbb{A}, r, \gamma)$, with one episode per disruption
event. At each iteration the agent selects a value of $k$, the solver evaluates
$\mathcal{P}^{\mathrm{reop}}_m(k)$ under a per-call limit $T_{\mathrm{solve}}$,
and the outcome updates the envelope, the front $\mathcal{F}^{*}$ and the state.

\paragraph{Initialization.}
Before the learning loop, the framework computes the minimal feasible deviation
$k_{\min}$ defined in \ref{eq:kmin_def}, by solving a MILP that minimizes the
$\ell_1$-distance to the reference plan subject to the rescheduling constraints,
the disruption restrictions, and the fixing of decisions completed before the
event. This solve is deterministic and lies
outside the policy: it anchors the front at its most stable extreme and
guarantees that one recovery option exists regardless of the severity of the
disruption.

A second quantity is computed once, in a single LP solve: the value
$z_{\mathrm{LP}}$ of the continuous relaxation at the upper limit $k_{\max}$.
Since $z_{\mathrm{LP}} \leq z^{*}(k)$ for every $k$, it is a valid floor for the
whole staircase and measures the improvement that may still remain. It is used
both as a state feature and in the termination logic: a search whose incumbent
has reached this floor cannot improve further, whereas a persistent gap indicates
that a larger value of $k$ may still pay off.

\paragraph{State space.}
At iteration $t$ the agent observes $s_t \in \mathbb{R}^{9}$,
\begin{equation}
    s_t = \Big[
        \tfrac{k_t}{k_{\min}},\;
        \tfrac{|z_{0} - z^{*}_{t}|}{|z_{0}|},\;
        \tfrac{t_{\mathrm{solve}}}{T_{\mathrm{solve}}},\;
        \tfrac{t_{\mathrm{elapsed}}}{T_{\mathrm{total}}},\;
        \mathbf{1}[\text{Pareto}],\;
        \mathbf{1}[\text{feasible}],\;
        \tfrac{|\mathcal{F}^{*}_{t}|}{\bar{N}},\;
        \phi_e,\;
        g_t
    \Big],
    \label{eq:state}
\end{equation}
where all components are normalized so that they vary on comparable ranges, which
is what makes the coefficients of the linear policy comparable to one another.

The first two features locate the search in the stability--cost space. The
expansion ratio $k_t / k_{\min}$ compares the current value of $k$ with the
stable anchor, and equals one at the first solve. The relative improvement
compares the incumbent $z^{*}_{t}$, the best objective found up to iteration $t$,
with $z_{0} = z^{*}(k_{\min})$, the objective of the anchor plan. The next two features describe the computational state. Here $t_{\mathrm{solve}}$
is the time consumed by the most recent solver call and $T_{\mathrm{solve}}$ the
limit imposed on each call, so their ratio measures the difficulty of the last
subproblem: a value close to one indicates that the solve was truncated.
Likewise, $t_{\mathrm{elapsed}}$ is the time consumed since the beginning of the
episode and $T_{\mathrm{total}}$ the budget allocated to the disruption event, so
their ratio reports how much of that budget remains and drives the agent to
concentrate exploration early. The two binary indicators state whether the last solve produced a new
non-dominated point and whether it returned a feasible solution; the distinction
matters because an infeasible solve carries no information about the staircase.
The seventh feature measures the coverage of the front, where
$|\mathcal{F}^{*}_{t}|$ is the number of non-dominated plans found up to
iteration $t$ and $\bar{N}$ a normalizing constant, the ratio being capped at
one. The last two features condition the qualitative behaviour of the policy. The
categorical encoding
$\phi_e \in \{0,\,0.25,\,0.5,\,0.75,\,1\}$ identifies the disruption type, in the
order road closure, truck breakdown, travel delay, demand increase and demand
decrease, so that a single policy can adapt to the nature of the event; localized
and systemic disruptions have markedly different productive regions. The relative
gap to the linear-programming floor $z_{\mathrm{LP}}$,
\begin{equation}
    g_t \;=\; \min\left\{ 
    \frac{z^{*}_{t} - z_{\mathrm{LP}}}{\max\{|z^{*}_{t}|,\,1\}},\;\; 1 \right\},
    \label{eq:gap_lp}
\end{equation}
is close to zero when the incumbent has reached the floor and no further step can
exist, and strictly positive when improvement potential remains. Without it, a
plateau of equal objective values is indistinguishable from an exhausted search;
this feature is what allows the agent to separate the two situations and to
choose between stopping and jumping further.

\paragraph{Action space.}
The agent selects among four actions,
\begin{equation}
    \mathbb{A} = \{\texttt{stop},\; \texttt{large-step},\;
                   \texttt{small-step},\; \texttt{bisect}\},
    \label{eq:actions}
\end{equation}
whose semantics are defined relative to the set of values already evaluated, so
that each action targets a region that has not yet been explored. Let
$\hat{k} = \max \mathcal{K}$ denote the largest value evaluated so far. The two
forward actions extend the sweep beyond $\hat{k}$,
\begin{equation}
    k_{t+1} =
    \begin{cases}
        \hat{k} + \delta_{\mathrm{s}}, & \text{if } a_t = \texttt{small-step},\\[2pt]
        \hat{k} + \delta_{\mathrm{L}}, & \text{if } a_t = \texttt{large-step},
    \end{cases}
    \qquad
    \delta_{\mathrm{s}} = \max\Bigl\{5,\ \bigl\lceil \tfrac{k_{\min}}{2}\bigr\rceil\Bigr\},
    \quad
    \delta_{\mathrm{L}} = \max\{2\delta_{\mathrm{s}},\ \bar{\tau}\},
    \label{eq:steps}
\end{equation}
where $\bar{\tau}$ is the average number of binary changes induced by a single
truck trip, estimated from the reference plan as the ratio of active arcs to
trips. Both increments are instance-adaptive: the small step refines the sweep at
the scale of the anchor, while the large step advances by approximately one trip,
the granularity at which the staircase actually changes.

The \texttt{bisect} action searches where a new non-dominated point may still
exist. Consider two consecutive evaluated values $k_1 < k_2$ with
$k_2 - k_1 > 1$. If $\bar{z}(k_1) > \bar{z}(k_2)$, the cost drops somewhere in
$(k_1, k_2]$, so a non-dominated point may lie strictly between them; we call
such a pair an \emph{active interval}. The action selects the active interval
with the largest cost drop and evaluates its midpoint, which halves the range of
values of $k$ in which the point can lie.

Finally, \texttt{stop} ends the episode. Making termination an explicit decision,
rather than a passive consequence of budget exhaustion, is what allows the agent
to learn when to conserve computation. All values are clipped to
$[k_{\min},\, k_{\max}]$, where $k_{\max}$ is the number of active arcs of the
reference plan, beyond which the local branching constraint is no longer binding.
\paragraph{Reward function.}
The reward is aligned with the objective of the search, namely the discovery of
many non-dominated plans as early as possible:
\begin{equation}
    r_t =
    \begin{cases}
        \rho_{\mathrm{P}}
        + \rho_{\mathrm{D}} \min\Bigl\{\dfrac{\Delta z_t}{\alpha\,|z^{\mathrm{ref}}|},\,1\Bigr\}
        + \rho_{\mathrm{T}}\dfrac{t_{\mathrm{solve}}}{T_{\mathrm{solve}}},
        & \text{new non-dominated point},\\[10pt]
        \rho_{\mathrm{W}} + \rho_{\mathrm{T}}\dfrac{t_{\mathrm{solve}}}{T_{\mathrm{solve}}},
        & \text{new solve, no new point},\\[10pt]
        \rho_{\mathrm{C}}, & \text{value already evaluated},
    \end{cases}
    \label{eq:reward}
\end{equation}
where $\Delta z_t$ is the reduction of the envelope obtained at iteration $t$ and
$z^{\mathrm{ref}}$ the objective of the reference plan before the disruption. A
solve that extends the front earns the reward $\rho_{\mathrm{P}}$, increased by a
bonus of magnitude at most $\rho_{\mathrm{D}}$ that grows with the cost
reduction, the scale $\alpha$ setting the drop above which the bonus saturates;
large steps of the staircase are thus valued above marginal ones. A solve that
returns no new point receives the penalty $\rho_{\mathrm{W}}$. In both cases the
reward also carries the term $\rho_{\mathrm{T}}$, weighted by the fraction of the
per-call limit consumed, which reflects the cost of the call and leads the agent
to explore early. If the agent proposes a value of $k$ already evaluated, the
recorded objective is returned without calling the solver, and the penalty
$\rho_{\mathrm{C}}$ discourages repetition at no computational cost.

Stopping is rewarded according to the evidence available when the decision is
taken:
\begin{equation}
    r^{\texttt{stop}} =
    \begin{cases}
        +\rho_{\mathrm{S}}, & \text{no active interval and } g_t \leq \varepsilon_{\mathrm{LP}},\\[2pt]
        0,                  & \text{no active interval and } g_t > \varepsilon_{\mathrm{LP}},\\[2pt]
        -\rho_{\mathrm{S}}, & \text{at least one active interval remains}.
    \end{cases}
    \label{eq:stop_reward}
\end{equation}
Stopping is therefore rewarded only when no interval remains to be searched and
the incumbent has reached the floor, is neutral when no interval remains but
improvement is still possible, and is penalized when an active interval remains,
which discourages premature termination. The reward $\rho_{\mathrm{P}}$, the
bonus $\rho_{\mathrm{D}}$, the penalties $\rho_{\mathrm{W}}$, $\rho_{\mathrm{C}}$
and $\rho_{\mathrm{T}}$, the termination magnitude $\rho_{\mathrm{S}}$, and the
scale $\alpha$ are reported in Table~\ref{tab:rl_hyperparams}.

\paragraph{Policy and optimization.}
The policy $\pi_\theta : \mathcal{S} \to \Delta^{|\mathbb{A}|}$ maps the state to
action logits through a single linear layer,
$\pi_\theta(s) = \mathrm{softmax}(W s + b)$ with $W \in \mathbb{R}^{4 \times 9}$
and $b \in \mathbb{R}^{4}$, initialized following
\cite{glorot2010understanding}. The architecture is deliberately parsimonious:
each solver call consumes minutes, so an episode yields only a handful of
transitions and a deeper network would overfit; the linear form also keeps each
coefficient interpretable as the influence of one feature on the propensity to
advance, bisect or stop. Actions are sampled from the induced categorical
distribution, which provides exploration without an explicit schedule.

Parameters are optimized with REINFORCE \cite{williams1992simple}. For an episode
$\tau = (s_0, a_0, r_0, \ldots, s_T)$ the discounted returns
$G_t = \sum_{u \geq t} \gamma^{\,u-t} r_u$ are standardized within the episode,
which acts as a variance-reduction baseline and matters here because episodes are
short and rewards are sparse. The loss augments the policy-gradient term with an
entropy bonus,
\begin{equation}
    \mathcal{L}(\theta) = -\sum_{t=0}^{T} \hat{G}_t \log \pi_\theta(a_t \mid s_t)
    \;-\; \beta \sum_{t=0}^{T} H\bigl(\pi_\theta(\cdot \mid s_t)\bigr),
    \label{eq:loss}
\end{equation}
which prevents the policy from collapsing prematurely onto the \texttt{stop}
action, a degenerate optimum that is otherwise attractive because stopping is
free whereas probing is not. Updates are applied by Adam \cite{kingma2014adam}
after each completed episode.

\paragraph{Termination.}
An episode ends when the agent selects \texttt{stop}; when the remaining budget
falls below the reserve required for one solve; when $N^{\mathrm{streak}}$
consecutive \emph{fresh} probes fail to extend the front, cached proposals being
excluded from the count so that the criterion measures genuine exploration; or
when the search is provably exhausted, that is, when $\hat{k} = k_{\max}$ and no
active interval remains.

\paragraph{Training approach.}
A separate policy is trained for each weekly instance, which matches the
operating rhythm of the partner: the tactical plan is released at the end of the
week, the agent is trained before execution begins, and the resulting policy is
available when disruptions occur during the week. Specializing the policy to the
instance is therefore operationally natural, and it allows the learned rule to
adapt to the structure of the staircase induced by that particular network and
fleet.

Training runs for $N_{\mathrm{epoch}} = 50$ epochs. A fresh scenario of
disruptions is regenerated at every epoch, with the generator seed advancing, so
that the agent is exposed to different events rather than replaying a fixed set;
this prevents overfitting to a particular scenario. The first epoch is run in
observation mode, without gradient updates, to collect baseline performance. Table~\ref{tab:rl_hyperparams} reports the
complete parameter setting.

\begin{table}[ht]
\centering
\caption{Parameter setting of the reinforcement learning agent.}
\label{tab:rl_hyperparams}
\small
\begin{tabular}{llc}
\toprule
\textbf{Symbol} & \textbf{Description} & \textbf{Value} \\
\midrule
$N_{\mathrm{epoch}}$ & Training epochs & $50$ \\
$T_{\mathrm{total}}$ & Time budget per disruption event (s) & $900$ \\
$T_{\mathrm{solve}}$ & Time limit per solver call (s) & $180$ \\
$N^{\mathrm{streak}}$ & Fresh non-improving probes before stopping & $6$ \\
$\delta_{\mathrm{s}}$ & Small step & $\max\{5, \lceil k_{\min}/2\rceil\}$ \\
$\delta_{\mathrm{L}}$ & Large step & $\max\{2\delta_{\mathrm{s}}, \bar{\tau}\}$ \\
$k_{\max}$ & Upper limit on $k$ & active arcs of $X^{*}$ \\
$\bar{N}$ & Normalization constant for the front size & $10$ \\
$\varepsilon_{\mathrm{LP}}$ & Tolerance on the gap to the LP floor & $10^{-4}$ \\
$\rho_{\mathrm{P}},\ \rho_{\mathrm{D}}$ & Pareto reward and drop bonus & $1.0$,\ $0.5$ \\
$\rho_{\mathrm{W}},\ \rho_{\mathrm{C}},\ \rho_{\mathrm{T}}$ & Wasted, cached, time penalties & $-0.3$,\ $-0.5$,\ $-0.1$ \\
$\rho_{\mathrm{S}}$ & Termination reward magnitude & $0.5$ \\
$\alpha$ & Scale of the drop bonus & $0.01$ \\
$\gamma$ & Discount factor & $0.99$ \\
$\beta$ & Entropy coefficient & $0.01$ \\
$\eta$ & Learning rate (Adam) & $10^{-2}$ \\
\bottomrule
\end{tabular}
\end{table}


\section{Computational Study}
\label{sec:experiments}

This section reports the computational experiments carried out to evaluate the proposed framework for the reoptimization of the $\mathcal{LTRSP}$ under unforeseen events. The objective of the study is to quantify the efficiency, adaptability, and resilience of the integrated methodology, as well as the relative contribution of its two main components: the mathematical programming approach based on local branching and the reinforcement learning module. All experiments were conducted on a 3.7~GHz 32-core CPU workstation with 128~GB RAM. The MILP components were solved with \texttt{Gurobi~12.0}, and the learning agent was implemented in \texttt{PyTorch~2.2}. 

The experimental plan follows a structured process inspired by resilience-oriented optimization studies. Each experiment begins with the generation of a nominal tactical plan $\hat{x}$ for a given weekly instance, solved under deterministic conditions with the base model $\mathcal{P}_m$. This nominal plan represents the optimal routing and scheduling configuration in the absence of disruptions. 

Once the baseline is established, a set of scenarios $\Omega = \{\mathcal{E}_1, \ldots, \mathcal{E}_{|\Omega|}\}$ is produced by the Unforeseen Events Generator described in Section~\ref{sec:methodology}. Each scenario $\mathcal{E}_\omega$ corresponds to a distinct realization of disruptions over the planning horizon, including the occurrence time, spatial location, and duration of each event. The Disruption Characterization and Impact Module (DCIM) then transforms each $\mathcal{E}_\omega$ into technical model modifications $(\Delta A_\omega, \Delta \theta_\omega)$, which are applied to the original network and parameter set. The resulting instance $\mathcal{P}_m(e)$ captures both the new feasibility conditions and the operational impacts of the disruptions.

\subsection{Instances and Unforeseen Events Generation}
\label{sec:instances_events}

The computational experiments are based on 20 weekly instances 
$\mathcal{W} = \{W_1, W_2, \ldots, W_{20}\}$ derived from real operational 
data provided by our industrial partner. Each instance corresponds to a 
one-week tactical plan of log-truck movements between forest blocks, mills, 
and home bases.

Following the classification proposed in prior large-scale vehicle routing 
studies, the instances are grouped according to the density of their 
routing network, expressed by the total number of arcs $|A|$. Specifically, 
small instances contain fewer than 30{,}000 arcs, medium instances include 
between 30{,}000 and 50{,}000 arcs, and large instances exceed 50{,}000 
arcs. The number of arcs serves as a proxy for network complexity and 
directly affects computational difficulty: denser networks yield 
exponentially more feasible routes, leading to a combinatorial explosion of 
potential trip combinations. This characteristic is particularly relevant 
in heterogeneous-fleet problems such as the $\mathcal{LTRSP}$, where 
vehicle types, capacities, and equipment configurations interact with the 
network structure.

Each instance includes up to five days of operations, multiple product 
categories, and time-dependent travel times calibrated from GPS and 
road-condition data. The baseline plan $X^*$ obtained from the 
deterministic optimization model $\mathcal{P}_m$ serves as the reference 
solution for all disruption experiments. The full characteristics of the 
20 weekly instances are reported in Table~\ref{instances-characteristics}.

\begin{table}[H]
\centering
\caption{Characteristics of the 20 weekly instances categorized by number 
of arcs. Instance: label of the weekly instance; Dates: planning horizon; 
Served Mills: number of mills served; Used Blocks: number of forest blocks 
used; Products: number of product types; Vehicles: number of vehicles 
available; Homebases: number of truck home bases; Nodes: number of nodes in 
the network; Arcs: number of arcs; Total Demand (GMT): total demand 
expressed in green metric tonnes; Avg Distance: average distance per arc 
(km); Max Distance: maximum distance between any two points (km).}
\resizebox{\textwidth}{!}{%
\begin{tabular}{cccccccccccc}
\hline
\textbf{Instance} & \textbf{Dates} & \textbf{Served Mills} & \textbf{Used Blocks} & \textbf{Products} & \textbf{Vehicles} & \textbf{Homebases} & \textbf{Nodes} & \textbf{Arcs} & \textbf{Total Demand (GMT)} & \textbf{Avg Distance} & \textbf{Max Distance} \\ \hline

\multicolumn{12}{c}{\textbf{Small Instances (Arcs $<$ 30,000)}} \\ \hline
$W_1$  & 2025-01-10 -- 2025-01-14 & 11 & 25 & 3 & 70 & 27 & 1295 & 26372 & 6140  & 242.06 & 574.00 \\ \hline
$W_2$  & 2024-12-30 -- 2025-01-04 & 13 & 24 & 3 & 60 & 19 & 1303 & 26504 & 8820  & 272.00 & 629.27 \\ \hline
$W_3$  & 2024-12-24 -- 2024-12-29 & 11 & 19 & 3 & 36 & 19 & 1063 & 17189 & 3820  & 259.31 & 651.28 \\ \hline
$W_4$  & 2024-11-29 -- 2024-12-03 & 13 & 23 & 3 & 51 & 25 & 1267 & 26934 & 7180  & 245.39 & 680.30 \\ \hline
$W_5$  & 2024-11-24 -- 2024-11-28 & 15 & 23 & 3 & 53 & 17 & 1321 & 29265 & 10440 & 206.22 & 628.94 \\ \hline

\multicolumn{12}{c}{\textbf{Medium Instances (30,000 $\leq$ Arcs $\leq$ 50,000)}} \\ \hline
$W_6$    & 2024-12-04 -- 2024-12-08 & 14 & 24 & 3 & 57 & 25 & 1333 & 32603 & 9600  & 231.08 & 660.30 \\ \hline
$W_7$    & 2024-12-09 -- 2024-12-13 & 17 & 30 & 3 & 66 & 27 & 1629 & 45237 & 14250 & 227.31 & 574.37 \\ \hline
$W_8$    & 2024-11-19 -- 2024-11-23 & 14 & 22 & 3 & 67 & 27 & 1267 & 32044 & 10800 & 283.70 & 657.06 \\ \hline
$W_9$    & 2024-11-14 -- 2024-11-18 & 17 & 23 & 3 & 68 & 30 & 1376 & 37120 & 10320 & 229.56 & 660.30 \\ \hline
$W_{10}$ & 2024-11-09 -- 2024-11-13 & 16 & 21 & 3 & 62 & 26 & 1280 & 31172 & 9630  & 225.54 & 657.06 \\ \hline
$W_{11}$ & 2024-10-30 -- 2024-11-03 & 18 & 46 & 4 & 59 & 27 & 1465 & 44237 & 12930 & 241.16 & 728.76 \\ \hline
$W_{12}$ & 2024-10-14 -- 2024-10-18 & 16 & 26 & 4 & 56 & 26 & 1464 & 36989 & 8140  & 264.96 & 739.14 \\ \hline
$W_{13}$ & 2024-10-10 -- 2024-10-14 & 13 & 37 & 4 & 62 & 27 & 1777 & 40795 & 11330 & 252.04 & 739.14 \\ \hline
$W_{14}$ & 2025-01-05 -- 2025-01-09 & 16 & 31 & 3 & 88 & 27 & 1637 & 43155 & 20130 & 244.92 & 574.00 \\ \hline

\multicolumn{12}{c}{\textbf{Large Instances (Arcs $>$ 50,000)}} \\ \hline
$W_{15}$ & 2024-10-25 -- 2024-10-29 & 17 & 37 & 4 & 72 & 30 & 1888 & 53103 & 12850 & 246.67 & 657.06 \\ \hline
$W_{16}$ & 2024-12-19 -- 2024-12-23 & 17 & 35 & 3 & 71 & 28 & 1808 & 50238 & 10540 & 268.47 & 651.28 \\ \hline
$W_{17}$ & 2024-12-14 -- 2024-12-18 & 17 & 38 & 3 & 64 & 27 & 1917 & 52736 & 11630 & 257.75 & 659.57 \\ \hline
$W_{18}$ & 2024-11-04 -- 2024-11-08 & 18 & 46 & 4 & 72 & 27 & 2233 & 86959 & 20240 & 220.54 & 702.36 \\ \hline
$W_{19}$ & 2024-10-20 -- 2024-10-24 & 17 & 40 & 4 & 78 & 30 & 1996 & 60449 & 23500 & 234.08 & 721.61 \\ \hline
$W_{20}$ & 2024-10-15 -- 2024-10-19 & 16 & 36 & 4 & 85 & 27 & 1825 & 50790 & 17170 & 227.70 & 739.14 \\ \hline

\end{tabular}%
}
\label{instances-characteristics}
\end{table}

To assess the robustness and adaptability of the proposed framework, each 
instance is subjected to a set of synthetically generated unforeseen events 
$\mathcal{E}_\omega$, designed to span the main classes of disruptions 
encountered in forest transportation systems. Following stochastic 
simulation principles widely used in resilient transportation and supply 
chain optimization~\citep{wang2023resilience}, the event-generation process 
explicitly accounts for randomness in occurrence, duration, location, and 
severity. Consistent with the framing of 
Section~\ref{sec:event-generator}, the distributions in 
Table~\ref{tab:event_distributions} are used here as a parameterised 
stress-test specification rather than as a calibrated probabilistic model 
of forestry incidents.

Formally, let $\Omega = \{\mathcal{E}_1, \ldots, \mathcal{E}_{|\Omega|}\}$ 
denote the set of disruption scenarios, where each scenario 
$\mathcal{E}_\omega$ is composed of a finite number of stochastic events:
\[
\mathcal{E}_\omega = \{e_1, e_2, \ldots, e_{n_\omega}\},
\qquad
e_i = (t_i, \; l_i, \; \tau_i,\; d_i).
\]

In line with common practice in the resilient scheduling literature ~\citep{gendreau201650th}, ten disruptions are generated per 
instance, corresponding to approximately two events per operational day. The stochastic mechanisms are summarised in 
Table~\ref{tab:event_distributions}.

The event generation follows a two-stage process combining temporal and 
spatial sampling. First, the number of events per day $N_t$ is sampled from 
a Poisson process,
\[
N_t \sim \mathrm{Poisson}(\lambda_t), \qquad t = 1,\ldots,T,
\]
where $\lambda_t$ denotes the expected daily frequency, set to two events 
per day in the experiments. Conditional on $N_t$, the occurrence time of 
each event within day $t$ is drawn from a continuous uniform distribution 
$\mathcal{U}(0,H)$ over the working hours $H = [6{:}00, 18{:}00]$. The 
duration $e_i$ and location $l_i$ are then drawn according to the 
corresponding event law in Table~\ref{tab:event_distributions}. Spatial 
dependencies are introduced for certain event types (e.g., road closures) 
through a correlated sampling mechanism: the occurrence probability of a 
closure on road $a $ is increased by a factor $\eta = 1.5$ if adjacent 
arcs have already been closed on the same day, reflecting the spatial 
propagation of weather-induced failures~\citep{flisberg2022optimized}.

\begin{table}[H]
\centering
\caption{Example of generated unforeseen events over a five-day horizon for 
a representative weekly instance.}
\label{tab:example_events}
\resizebox{\textwidth}{!}{
\begin{tabular}{cccccccc}
\hline
\textbf{Day} & \textbf{Event ID} & \textbf{Event Type} & \textbf{Location} & \textbf{Start Time} & \textbf{Duration} & \textbf{Distribution Source} & \textbf{Impact on Model} \\ \hline
1 & $e_1$    & Road closure     & Road (F23 $\rightarrow$ HB07) & 07:30 & 2 days   & Bernoulli($p=0.08$) + $U[1,3]$         & Arc removal $\Delta A(e_1)$ \\
1 & $e_2$    & Demand decrease  & Mill M12                      & 09:00 & 1 day    & $\mathcal{N}(0,0.1^2)$                 & $\Delta d_{mp} = -8\%$ \\
2 & $e_3$    & Truck breakdown  & Vehicle V34                   & 10:15 & 6 hours  & Poisson($\lambda=0.05$) + $U[4,12]$    & Vehicle unavailable $\Delta \theta(e_3)$ \\
2 & $e_4$    & Travel delay     & Arc (HB05 $\rightarrow$ M03)  & 13:45 & 3 hours  & Exp($\lambda=0.15$)                    & $\tau_{ij} \leftarrow \tau_{ij} + \Delta\tau$ \\
3 & $e_5$    & Road closure     & Road (F41 $\rightarrow$ M09)  & 08:10 & 1 day    & Bernoulli($p=0.08$) + $U[1,3]$         & Arc removal $\Delta A(e_5)$ \\
3 & $e_6$    & Demand increase  & Mill M15                      & 15:00 & 1 day    & $\mathcal{N}(0,0.1^2)$                 & $\Delta d_{mp} = +12\%$ \\
4 & $e_7$    & Truck breakdown  & Vehicle V05                   & 11:20 & 10 hours & Poisson($\lambda=0.05$) + $U[4,12]$    & Vehicle downtime $\Delta \theta(e_7)$ \\
4 & $e_8$    & Travel delay     & Arc (F06 $\rightarrow$ HB03)  & 14:30 & 2 hours  & Exp($\lambda=0.15$)                    & Shift in time windows \\
5 & $e_9$    & Road closure     & Road (HB01 $\rightarrow$ M02) & 06:45 & 2 days   & Bernoulli($p=0.08$) + $U[1,3]$         & $\Delta A(e_9)$, infeasible trips pruned \\
5 & $e_{10}$ & Demand decrease  & Mill M04                      & 09:00 & 1 day    & $\mathcal{N}(0,0.1^2)$                 & $\Delta d_{mp} = -5\%$ \\ \hline
\end{tabular}
}
\label{example-events}
\end{table}

The resulting set of events $\mathcal{E}_\omega$ thus spans both isolated 
and cascading disruptions across space and time. All events are sorted chronologically 
to build a dynamic timeline as presented in the example of table \ref{example-events}. This stochastic construction yields diverse 
scenarios, ranging from single short-duration events to concurrent 
multi-day disruptions, covering the operational range of interest for the 
resilience evaluation.

Each event is then passed to the Disruption Characterization and Impact 
Module (DCIM), which transforms it into a pair of model perturbations 
$(\Delta A(e), \Delta \theta(e))$. This transformation ensures that 
disruptions are encoded directly within the mathematical structure of the 
MILP model: road closures modify the feasible network by removing affected 
arcs, breakdowns reduce available vehicle capacity over specific time 
windows, delays shift time-dependent travel durations, and demand 
variations alter the right-hand side of demand constraints.

\subsection{Numerical results}
\label{sec:results}
We evaluate the proposed framework on 20 weekly instances
$\mathcal{W} = \{W_1, \ldots, W_{20}\}$ derived from real operational data
provided by our Canadian forestry partner, grouped into small
($|A| < 30{,}000$), medium ($30{,}000 \leq |A| \leq 50{,}000$), and large
($|A| > 50{,}000$) categories. Each instance is subjected to 10 synthetically
generated disruption scenarios spanning road closures, truck breakdowns, travel
delays, and demand variations, following the generation rules of Table \ref{tab:event_distributions}, which
delays, and demand variations, following the generation rules of Table , which
yields 200 recovery problems per strategy.

All strategies begin identically: the disruption is characterised, $k_{\min}$ is
computed, and the local branching model is solved at $k = k_{\min}$ to obtain the
first Pareto point, the most stable feasible recovery plan. They then differ only
in the rule that selects the next value of $k$:
\begin{enumerate}
    \item \textbf{Fixed-$k$ grid.} The interval $[k_{\min}, k_{\max}]$ is divided
    into equally spaced steps and $\mathcal{P}^{\mathrm{reop}}_m(k)$ is solved at
    each grid value in turn until the budget $T_{\mathrm{total}}$ is exhausted.
    The grid is set independently of what previous solves reveal, which makes it
    the natural reference for the effort required without any adaptive rule. Additional details are given in \ref{app:baselines}.

    \item \textbf{Dichotomic search.} A deterministic rule that repeatedly splits
    the active interval with the largest cost drop, advances beyond the largest
    value already evaluated when no active interval remains, and stops when
    neither operation applies. This strategy shares the entire computational
    infrastructure of the learned policy, namely the memory of evaluated values,
    the monotone envelope, the elimination of flat intervals and the warm start.
    It differs from it only in the selection rule, so the comparison isolates the
    contribution of learning from that of the structure of the front. Further details are presented in Appendix \ref{app:baselines}.

    \item \textbf{RL-guided.} The trained agent selects $k$ from the
    9-dimensional state vector of Equation~\eqref{eq:state}, concentrating solver
    calls where new non-dominated plans are most likely, jumping across plateaus
    when the gap to the floor indicates remaining potential, and terminating when
    the evidence shows that no further plan can be found.
\end{enumerate}
All strategies operate under the same budget $T_{\mathrm{total}}$ per event and
the same solver settings, so differences in the reported metrics are attributable
to the selection rule alone. Since the 200 recovery problems are solved by every
strategy, the metrics are compared pairwise on each event, and the significance
of the differences is assessed with the Wilcoxon signed-rank test.

\subsubsection{Training Convergence}

Figure~\ref{fig:rl_training} reports training metrics over 50 epochs. The
average return increases steadily from near zero to approximately 6.5,
confirming that the REINFORCE agent learns to select actions that produce
Pareto-improving solutions. The average number of Pareto points per event
rises from 1.0 at epoch~0 to approximately 6.7 at epoch~50, more than a
six-fold increase in the number of recovery options presented to the
dispatcher.
\begin{figure}[ht]
    \centering
    \includegraphics[width=\textwidth]{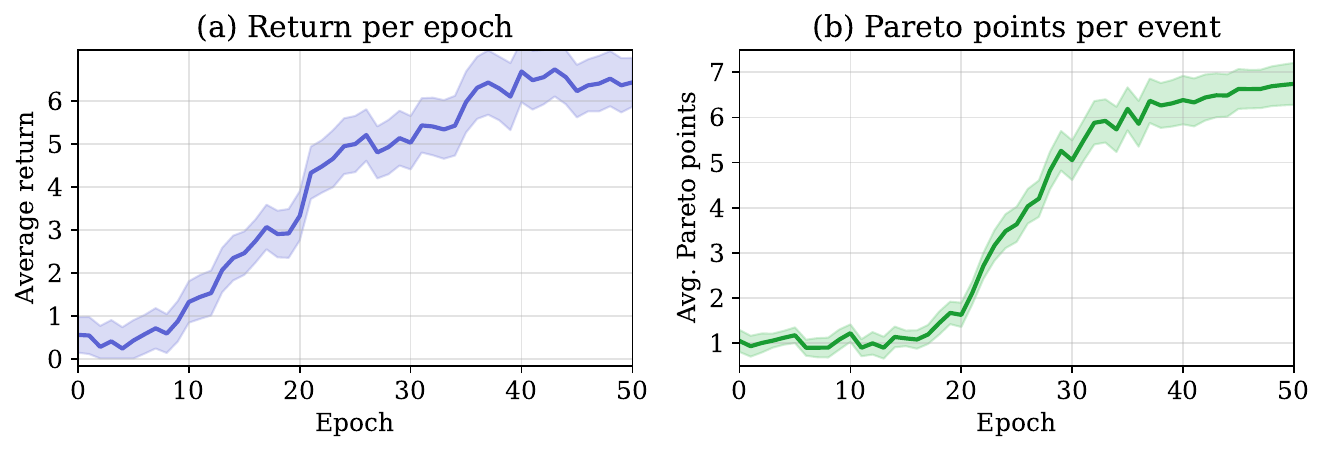}
    \caption{RL policy training over 50 epochs. Solid lines show the
    rolling mean; shaded regions indicate standard deviation.
    (a)~Average return per epoch. (b)~Average number of Pareto-optimal
    recovery plans discovered per event.}
    \label{fig:rl_training}
\end{figure}
The return remains close to zero during the first few epochs, then grows
steadily as the agent improves its action-selection strategy. The
convergence plateau is reached after approximately 35--40 epochs,
consistent with the small number of state-action observations per episode
and the linear policy architecture; beyond this point both metrics
stabilise and the variability across runs narrows. The final learned policy exhibits a structured adaptation rule: it favours
\texttt{small-step} advances for minor disruptions such as single travel
delays, and combines \texttt{large-step} jumps with \texttt{bisect}
refinements for systemic events such as multi-day road closures, rapidly
leaving the neighbourhood of $k_{\min}$ to explore more distant regions of
the deviation axis.
\subsubsection{Performance metrics}
\label{sec:metrics}

All strategies operate under the same time budget $T_{\mathrm{total}}$ per
disruption event and share the same computational infrastructure, so the
elapsed time is not itself a discriminating quantity. What separates the
strategies is how they convert this budget into recovery options and how good
those options are. We measure this along three complementary dimensions: the
computational efficiency of the exploration, the operational quality of the best
recovered plan, and the overall spread of the front.

\paragraph{Hit rate.}
The number of Pareto-improving calls per solver call,
\begin{equation}
    \eta \;=\; \frac{|\mathcal{F}^{*}|}{N_{\mathrm{solve}}} \;\in\; (0,1],
    \label{eq:hit_rate}
\end{equation}
where $N_{\mathrm{solve}}$ is the number of solver calls issued for the event
and $|\mathcal{F}^{*}|$ the size of the returned front. Since every plan on the
front originates from exactly one call, $\eta$ is the fraction of the
computational effort that translates into an option for the dispatcher.
$N_{\mathrm{solve}}$ includes the anchor solve at $k_{\min}$, which yields the
first point, and excludes proposals of a value already evaluated, which are
served from memory. A hit rate of one is not attainable in general: certifying
that a region of the staircase contains no further step requires at least one
call that returns none. This metric captures how efficiently the strategy
identifies productive regions of the deviation axis; a strategy that solves in
regions where the staircase is flat wastes calls and reports a low $\eta$.

\paragraph{Cost deviation.}
The relative change in pure routing cost of the best plan of the front, compared
to the baseline,
\begin{equation}
    \Delta\tilde{C} \;=\;
    \frac{\tilde{C}(x^*) - \tilde{C}(\hat{x})}{\tilde{C}(\hat{x})}
    \times 100\%,
    \label{eq:cost_dev}
\end{equation}
where $\tilde{C}(x^*)$ is the lowest pure routing cost attained across the
recovered Pareto points and $\tilde{C}(\hat{x})$ that of the baseline plan. A
value of zero means the disruption has been absorbed at no additional cost, and
a large positive value means the strategy failed to find a plan close to the
pre-disruption cost. Unlike the hit rate, which is agnostic to solution
quality, $\Delta\tilde{C}$ measures the operational impact actually incurred by
the company after recovery: it answers the question of whether the strategy
found a good plan, not merely a lot of plans.

\paragraph{Hypervolume.}
The hit rate rewards a strategy for finding many plans, but says nothing about
how those plans are distributed on the front, and $\Delta\tilde{C}$ reports only
the cheapest plan. Neither captures the case in which a strategy recovers a
large front made of nearby, mutually similar plans: many points, few genuine
options. The standard indicator that penalises this is the hypervolume
\citep{zitzler1999multiobjective}. In the $(k, z)$ plane, where both objectives
are minimised, it is the area dominated by $\mathcal{F}^{*}$ up to a reference
point, normalised by the area of the region in which any front must lie:
\begin{equation}
    HV(\mathcal{F}^{*}) \;=\;
    \frac{1}{(k_{\max} - k_{\min})\,(z^{\mathrm{w}}_e - z_{\mathrm{LP}})}
    \left|
        \bigcup_{(k,\,z)\, \in\, \mathcal{F}^{*}}
        [\,k,\, k_{\max}\,] \times [\,z,\, z^{\mathrm{w}}_e\,]
    \right|,
    \label{eq:hypervolume}
\end{equation}
where $|\cdot|$ denotes area and $z^{\mathrm{w}}_e$ is the highest cost observed
for event $e$ across all strategies, so that the same reference point is used
for all methods on a given event. The lower edge of the box is the
linear-programming floor $z_{\mathrm{LP}}$ of Section~\ref{sec:rl_approach},
which no plan can undercut. The resulting value lies in $[0,1]$ and grows both
with the number of recovered plans and with their coverage of the
stability--cost region: a front of ten mutually close plans is penalised
relative to a front of five well-separated ones. The interest of $HV$ is
comparative, since it summarises the front on a single scale on which
strategies producing fronts of different cardinality can be ranked.
\subsubsection{Comparison of Exploration Strategies}
\label{sec:comparison}

Each of the 20 weekly instances of $\mathcal{W}$ is subjected to the same
protocol as during training: a scenario of 10 disruption events is drawn from
the generator of Section~4.1, spanning road closures, truck breakdowns, travel
delays, and demand variations, so that the reported results cover the full
range of event categories rather than a fixed configuration.
Table~\ref{tab:comparison} reports the three metrics of Section~\ref{sec:metrics}
averaged over the 10 events of each instance, for the three strategies compared
at equal time budget. Training a dedicated policy on each instance takes
approximately 25~hours of wall-clock time on the hardware described in
Section~5.1, comfortably within the operational window between the release of
the weekly plan and the start of execution.

\begin{table}[t]
\centering
\caption{Comparison of the three search strategies.}
\label{tab:comparison}
\small
\setlength{\tabcolsep}{4pt}
\begin{tabular}{l ccc ccc ccc}
\toprule
& \multicolumn{3}{c}{\textbf{Fixed-$k$ grid}}
& \multicolumn{3}{c}{\textbf{Dichotomic search}}
& \multicolumn{3}{c}{\textbf{RL-guided}} \\
\cmidrule(lr){2-4}\cmidrule(lr){5-7}\cmidrule(lr){8-10}
\textbf{Instance}
& $\eta$ & $\Delta\tilde{C}$ & $HV$
& $\eta$ & $\Delta\tilde{C}$ & $HV$
& $\eta$ & $\Delta\tilde{C}$ & $HV$ \\
\midrule
\multicolumn{10}{l}{\textit{Small instances} ($|A| < 30{,}000$)} \\
\midrule
$W_1$    & 26.55 & 4.42 & 0.12 & 43.03 & 5.24 & 0.12 & 57.72 & 4.03 & 0.24 \\
$W_2$    & 31.37 & 6.44 & 0.17 & 52.18 & 8.03 & 0.19 & 82.24 & 3.45 & 0.31 \\
$W_3$    & 33.60 & 6.24 & 0.17 & 50.33 & 5.28 & 0.18 & 79.93 & 3.09 & 0.31 \\
$W_4$    & 37.44 & 7.36 & 0.15 & 52.04 & 4.65 & 0.16 & 63.11 & 3.07 & 0.35 \\
$W_5$    & 29.48 & 5.01 & 0.21 & 40.73 & 9.20 & 0.24 & 57.47 & 2.03 & 0.26 \\
\midrule
\multicolumn{10}{l}{\textit{Medium instances} ($30{,}000 \leq |A| \leq 50{,}000$)} \\
\midrule
$W_6$    & 38.23 & 8.26 & 0.12 & 45.46 & 9.96 & 0.13 & 59.41 & 3.69 & 0.38 \\
$W_7$    & 37.63 & 5.77 & 0.15 & 47.13 & 6.65 & 0.16 & 68.33 & 2.58 & 0.29 \\
$W_8$    & 35.70 & 7.42 & 0.14 & 46.94 & 9.31 & 0.15 & 59.36 & 4.25 & 0.35 \\
$W_9$    & 38.84 & 6.66 & 0.23 & 45.94 & 9.62 & 0.27 & 75.13 & 4.90 & 0.32 \\
$W_{10}$ & 30.28 & 8.38 & 0.19 & 53.41 & 7.58 & 0.22 & 74.93 & 2.43 & 0.38 \\
$W_{11}$ & 28.60 & 4.39 & 0.24 & 44.14 & 7.16 & 0.29 & 58.85 & 4.85 & 0.24 \\
$W_{12}$ & 28.92 & 8.22 & 0.21 & 53.98 & 4.45 & 0.25 & 56.72 & 3.78 & 0.38 \\
$W_{13}$ & 31.83 & 6.83 & 0.22 & 52.91 & 9.01 & 0.26 & 75.34 & 2.29 & 0.33 \\
$W_{14}$ & 27.58 & 4.54 & 0.12 & 53.09 & 5.75 & 0.13 & 80.81 & 3.66 & 0.25 \\
\midrule
\multicolumn{10}{l}{\textit{Large instances} ($|A| > 50{,}000$)} \\
\midrule
$W_{15}$ & 38.41 & 7.03 & 0.24 & 52.64 & 6.40 & 0.29 & 82.08 & 2.30 & 0.33 \\
$W_{16}$ & 27.08 & 4.01 & 0.13 & 48.61 & 5.53 & 0.14 & 77.84 & 1.98 & 0.22 \\
$W_{17}$ & 30.49 & 3.15 & 0.14 & 49.96 & 4.40 & 0.14 & 75.91 & 2.87 & 0.20 \\
$W_{18}$ & 38.87 & 8.76 & 0.18 & 44.45 & 6.68 & 0.19 & 70.22 & 4.71 & 0.39 \\
$W_{19}$ & 39.88 & 8.85 & 0.15 & 44.60 & 9.36 & 0.17 & 77.71 & 4.22 & 0.40 \\
$W_{20}$ & 30.42 & 5.06 & 0.12 & 42.43 & 6.74 & 0.12 & 78.89 & 2.82 & 0.27 \\
\midrule
\textbf{Average}
& \textbf{32.7\%} & \textbf{6.34\%} & \textbf{0.168}
& \textbf{48.2\%} & \textbf{7.05\%} & \textbf{0.195}
& \textbf{70.6\%} & \textbf{3.35\%} & \textbf{0.310} \\
\bottomrule
\end{tabular}
\end{table}

The three metrics point in the same direction across the twenty instances. The
learned policy attains the highest hit rate on every instance, averaging
$70.6\%$ against $48.2\%$ for the dichotomic search and $32.7\%$ for the
fixed-$k$ grid. Since the three strategies operate under identical time budgets
and share the same computational infrastructure -- the cache of evaluated
values, the monotone envelope, the closure of flat intervals, and the warm
start -- this difference is attributable to the selection rule alone. The
learned rule allocates the same effort to regions of the deviation axis where
a Pareto-improving solve is more likely to be found, while the two baselines
either search rigidly (fixed grid) or split only intervals already bracketed
between known points (dichotomic search).

The cost deviation of the best recovered plan follows the same ranking, falling
from $6.34\%$ and $7.05\%$ for the two baselines to $3.35\%$ for the learned
policy. This is the metric most directly interpretable in operational terms:
after recovery, the dispatcher retains, on average, a plan whose routing cost
is within $3.35\%$ of the pre-disruption baseline, against nearly $7\%$ under a
rule that does not exploit information from previous solves. The normalised
hypervolume separates the strategies on a single scale that combines quantity
and quality of plans, and rises consistently across the three columns
($0.168 \to 0.195 \to 0.310$), which indicates that the additional plans
returned by the learned policy do not accumulate in a narrow region of the
front but effectively widen its coverage of the stability--cost region.

The advantage is stable across the three instance size categories. On the six
large instances ($|A| > 50{,}000$), the learned policy averages a hit rate of
$65.5\%$ and a hypervolume of $0.302$, comparable to the values on smaller
instances, which indicates that the advantage of adaptive selection does not
depend on the scale of the underlying network. Two instances warrant a specific
comment. On $W_{11}$, the learned policy reports the highest hit rate
($50.02\%$) but a lower hypervolume ($0.24$) than the dichotomic search
($0.29$): more Pareto-improving solves were found, but distributed in a
narrower region of the front, while the dichotomic search produced fewer plans
more widely spread. On $W_{18}$, the cost deviation of the learned policy
($4.71\%$) is slightly worse than that of the dichotomic search ($6.68\%$), yet
its hypervolume is more than twice as large ($0.39$ vs $0.19$): the policy
covered more of the front but did not reach its lowest point. These two cases
show that the three metrics are not perfectly correlated, which is precisely
why they are reported jointly: a strategy that dominates on one may lose on
another, and the joint reporting allows the reader to assess the trade-off
directly.

Figure~\ref{fig:hypervolume_illustration} makes the interpretation of the
hypervolume tangible on a single event. Both panels use the same reference box
$[k_{\min}, k_{\max}] \times [z_{\mathrm{LP}}, z^{\mathrm{w}}_e]$, so that the
dominated areas of the two fronts are directly comparable.

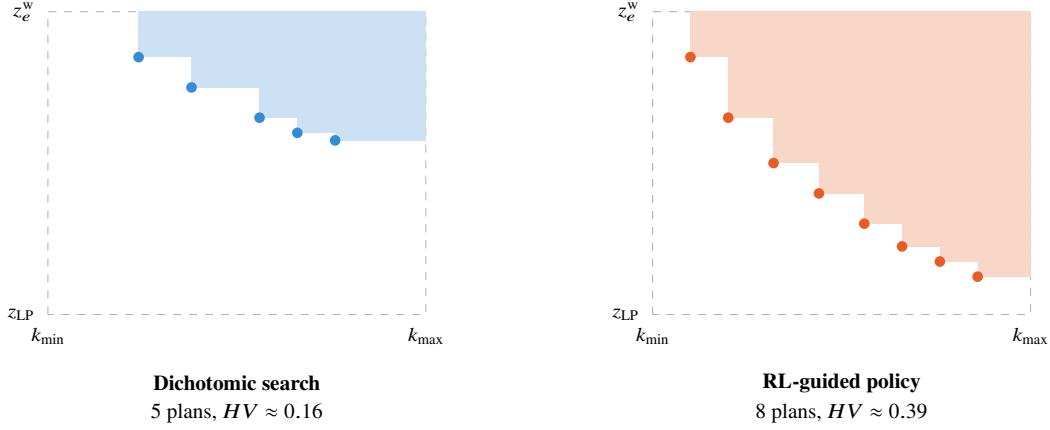
\begin{figure}[t]
\centering
\begin{tikzpicture}[scale=1.0, every node/.style={font=\footnotesize}]
\definecolor{bisFill}{RGB}{53, 139, 212}
\definecolor{rlFill}{RGB}{232, 93, 36}

\begin{scope}[xshift=0cm]
  \draw[dashed, gray!60] (0,0) rectangle (5,4);
  \node[below] at (0,-0.05) {\scriptsize $k_{\min}$};
  \node[below] at (5,-0.05) {\scriptsize $k_{\max}$};
  \node[left]  at (-0.05,0) {\scriptsize $z_{\mathrm{LP}}$};
  \node[left]  at (-0.05,4) {\scriptsize $z^{\mathrm{w}}_e$};

  \fill[bisFill!25] (1.2,3.4) rectangle (5,4);
  \fill[bisFill!25] (1.9,3.0) rectangle (5,3.4);
  \fill[bisFill!25] (2.8,2.6) rectangle (5,3.0);
  \fill[bisFill!25] (3.3,2.4) rectangle (5,2.6);
  \fill[bisFill!25] (3.8,2.3) rectangle (5,2.4);

  \foreach \x/\y in {1.2/3.4, 1.9/3.0, 2.8/2.6, 3.3/2.4, 3.8/2.3} {
    \fill[bisFill] (\x,\y) circle (2pt);
  }

  \node at (2.5,-0.9) {\textbf{Dichotomic search}};
  \node at (2.5,-1.3) {5 plans, $HV \approx 0.16$};
\end{scope}

\begin{scope}[xshift=8cm]
  \draw[dashed, gray!60] (0,0) rectangle (5,4);
  \node[below] at (0,-0.05) {\scriptsize $k_{\min}$};
  \node[below] at (5,-0.05) {\scriptsize $k_{\max}$};
  \node[left]  at (-0.05,0) {\scriptsize $z_{\mathrm{LP}}$};
  \node[left]  at (-0.05,4) {\scriptsize $z^{\mathrm{w}}_e$};

  \fill[rlFill!25] (0.5,3.4) rectangle (5,4);
  \fill[rlFill!25] (1.0,2.6) rectangle (5,3.4);
  \fill[rlFill!25] (1.6,2.0) rectangle (5,2.6);
  \fill[rlFill!25] (2.2,1.6) rectangle (5,2.0);
  \fill[rlFill!25] (2.8,1.2) rectangle (5,1.6);
  \fill[rlFill!25] (3.3,0.9) rectangle (5,1.2);
  \fill[rlFill!25] (3.8,0.7) rectangle (5,0.9);
  \fill[rlFill!25] (4.3,0.5) rectangle (5,0.7);

  \foreach \x/\y in {0.5/3.4, 1.0/2.6, 1.6/2.0, 2.2/1.6, 2.8/1.2, 3.3/0.9, 3.8/0.7, 4.3/0.5} {
    \fill[rlFill] (\x,\y) circle (2pt);
  }

  \node at (2.5,-0.9) {\textbf{RL-guided policy}};
  \node at (2.5,-1.3) {8 plans, $HV \approx 0.39$};
\end{scope}

\end{tikzpicture}
\caption{Illustration of the normalised hypervolume on a single disruption
event. In each panel, the dashed box is the reference region
$[k_{\min}, k_{\max}] \times [z_{\mathrm{LP}}, z^{\mathrm{w}}_e]$; filled
circles are the non-dominated plans returned by the strategy; the shaded area
is the region dominated by the front. The normalised hypervolume is the ratio
of the shaded area to the area of the reference box. Both strategies start
from the same anchor at $k_{\min}$ and operate under the same solver budget;
only the selection rule differs.}
\label{fig:hypervolume_illustration}
\end{figure}

The two panels of Figure~\ref{fig:hypervolume_illustration} share the same
reference box because the anchor $k_{\min}$, the upper limit $k_{\max}$, the
linear-programming floor $z_{\mathrm{LP}}$, and the worst observed cost
$z^{\mathrm{w}}_e$ are common to all strategies on a given event: they depend
on the event and the model, not on the search rule. The dichotomic search
returns five plans clustered near the anchor and terminates with its best cost
still well above the floor. Its dominated region is the small staircase at the
upper-left of the left panel, covering roughly $16\%$ of the box. The learned
policy, in contrast, does not remain confined to the region bracketed by its
early solves: when the gap to the floor indicates that improvement remains
possible while no active interval is left to split, it performs a long jump,
reaches values of $k$ that the dichotomic search never visits, and recovers
plans that approach the linear-programming floor. Its dominated region is the
much wider staircase of the right panel, covering roughly $39\%$ of the box.
This ratio of $2.4$ between the two hypervolumes is not a difference in the
number of plans alone -- it reflects the additional depth of the front, that
is, the ability to reach lower-cost plans that a purely interval-splitting
rule cannot produce by construction. The mechanism is the same one that
delivers the aggregate advantage in cost deviation and hypervolume observed in
Table~\ref{tab:comparison}.
\subsection{Discussion}

The results reported in the previous section demonstrate that the proposed
framework achieves its core operational objective: providing forestry
dispatchers with a structured set of recovery options within a time frame
compatible with real-time decision-making. The ability to recover, on average,
several non-dominated plans per event -- each representing a distinct
stability--cost trade-off -- fundamentally changes how planners can respond to
unforeseen events. Rather than being presented with a single prescribed
recovery schedule, dispatchers retain full decision authority and can select
the plan that best matches their immediate operational context, whether that
means preserving driver commitments and mill delivery schedules at the cost of
slightly higher routing costs, or accepting a more substantial restructuring
in exchange for meaningful cost savings.

The structural guarantee provided by $k_{\min}$ is of particular practical
value. The first Pareto point, always obtained at the minimum feasible
deviation from the baseline plan, ensures that at least one recovery option is
available within a single anchor call, regardless of the severity of the
disruption. This conservative anchor gives dispatchers a safe baseline that
requires the fewest changes to the pre-established schedule, which is
especially important in forestry operations where drivers, loader operators,
and mill staff organise their activities around the initial plan and where
last-minute changes carry coordination costs that are not fully captured by
the routing objective alone.

The efficiency of the REINFORCE-guided exploration is equally significant
from an operational standpoint. Under an identical time budget of 15 minutes
per event, the learned policy converts $70.6\%$ of its solver calls into
Pareto-improving plans, against $48.2\%$ for the dichotomic search and
$32.7\%$ for the fixed-$k$ grid, while reducing the cost deviation of the best
recovered plan from around $7\%$ to $3.35\%$. Since the three strategies share
the same infrastructure -- the cache of evaluated values, the monotone
envelope, the closure of flat intervals, and the warm start -- this difference
is attributable to the selection rule alone, and it leaves dispatchers with a
substantially larger and better distributed set of options to evaluate before
committing to a recovery. In a context where a truck breakdown on a remote
single-access road can quickly cascade into missed mill delivery windows and
loader idle time, the speed of recovery is as important as its quality.

\paragraph{Practical fit with the weekly planning cycle.}
Training a dedicated policy for each weekly instance may appear to introduce a
significant computational overhead: approximately 25 hours of wall-clock time
on the hardware described in Section~5.1. In practice, this overhead is
absorbed by the operational cycle itself. Tactical planning is completed on
Friday, and the plan is released to dispatchers over the weekend before
execution begins on Monday. During this idle window, no productive solver work
is required, and the policy for the coming week can be trained in the
background. When execution starts on Monday and disruptions begin to occur,
the trained policy is already available and its evaluation at decision time is
instantaneous: a single forward pass of the linear model returns the next
value of $k$, and the only computational cost incurred during recovery is the
local branching solve itself -- the same cost that all three strategies would
incur. Under this arrangement, per-instance training is not a limitation but a
natural fit with the weekly rhythm of forestry operations.

\medskip
The practical deployment of such a framework, however, must account for
several limitations specific to the Canadian forestry sector. First, the
framework assumes that the occurrence time, location, and duration of each
disruption are known at the moment of detection. In practice, the degree of
digitalisation varies considerably across forestry operations. Many companies
still rely on radio communication and manual reporting, which can introduce
delays between the actual occurrence of an event and its formal
characterisation in the system. Under such conditions, the disruption impact
module may operate on incomplete or imprecise information, potentially
leading to recovery plans that are suboptimal once the full extent of the
disruption becomes clear.

Second, the framework assumes that the real-time positions and load status
of all trucks are continuously available, as would be provided by an active
GPS fleet management system. While such systems are increasingly common in
large forestry operations, smaller contractors may not have this
infrastructure in place, limiting the accuracy of the pre-disruption plan
fixing constraints. Third, the stochastic disruption scenarios used in the
computational experiments, while calibrated through discussions with our
industrial partner, cannot fully replicate the spatial correlations and
seasonal patterns of real disruption occurrences, particularly those driven
by weather events that tend to affect multiple roads and sites simultaneously.

Despite these limitations, the framework is designed to be robust to
incomplete information: the local branching constraint ensures that any
feasible recovery plan remains close to the baseline regardless of whether
the disruption characterization is exact, and the Pareto front structure
allows planners to manually favour more conservative options when uncertainty
about the disruption's true extent is high. Future work should investigate
the integration of the framework with existing fleet management and road
monitoring systems, as well as the extension of the stochastic event
generator to incorporate spatial and temporal correlation structures derived
from historical incident data. The development of online learning mechanisms
that update the REINFORCE policy as new disruption patterns emerge over the
course of an operational season also represents a promising direction for
increasing the long-term adaptability of the framework.
\section{Conclusion}
\label{sec:conclusion}
This paper addressed the reoptimization of log-truck routing and scheduling
operations in the Canadian forestry industry, where unforeseen disruptions
frequently invalidate pre-established tactical plans and require rapid
recovery decisions under competing operational objectives.
We proposed an integrated framework that couples a mixed-integer linear
programming formulation of the $\mathcal{LTRSP}$ with a local branching
heuristic whose neighbourhood radius $k$ is adaptively controlled by a
reinforcement learning agent trained via the REINFORCE policy-gradient
algorithm. The framework departs from the conventional single-solution
recovery paradigm by constructing a Pareto front of non-dominated recovery
plans for each disruption event, providing dispatchers with an explicit menu
of stability--cost trade-offs from which the most contextually appropriate
option can be selected. The minimum feasible deviation $k_{\min}$ is computed
exactly prior to each exploration, anchoring the Pareto front at its most
stable extreme and guaranteeing the availability of at least one feasible
recovery plan within a single solver call.

The computational evaluation on twenty weekly instances derived from real
operational data of a Canadian forestry partner demonstrated that the
RL-guided framework consistently recovers richer Pareto fronts than both a
fixed-$k$ grid baseline and a dichotomic search across all disruption types
and network size categories, while operating under identical time budgets.
Since the three strategies share the same computational infrastructure -- the
cache of evaluated values, the monotone envelope, the closure of flat
intervals, and the warm start -- the difference is attributable to the
selection rule alone, and originates in the ability of the learned policy to
leave the region of the deviation axis bracketed by known Pareto points when
the gap to the linear-programming floor indicates that improvement remains
possible. Cost deviations remain controlled and plan stability is preserved
across all tested configurations, confirming that the framework achieves
effective cost recovery while maintaining the continuity of driver schedules,
loader operations, and mill deliveries. The per-instance training cycle
naturally aligns with the weekly rhythm of forestry operations, since the
policy for the coming week can be trained during the weekend and is ready
before execution begins on Monday. These results validate the practical
viability of the approach as a decision-support tool for large-scale forestry
logistics operations.

Several directions remain open for future research. The deployment of the
framework within a fully digitised operational environment, with direct
integration to fleet management and road monitoring systems, would allow
disruption information to flow automatically into the characterization module,
reducing the dependency on manual reporting that currently limits
responsiveness in many Canadian forestry operations. The generalisation of the
Pareto front exploration paradigm to other transportation reoptimization
settings, including rail freight, intermodal logistics, and humanitarian
supply chains, also represents a natural and promising extension of the
contributions presented here.

\appendix
\section{Baseline exploration strategies}
\label{app:baselines}

This appendix describes the two deterministic baselines used in the
computational comparison of Section~\ref{sec:comparison}: the dichotomic search
and the fixed-$k$ grid. Both operate within the same computational
infrastructure as the RL-guided policy -- the cache of evaluated values, the
monotone envelope $\bar{z}(k)$, the closure of flat intervals, and the warm
start of Section~\ref{sec:rl_approach} -- and differ from it only in the rule
that selects the next value of $k$. This design isolates the contribution of
the learned selection rule from that of the shared search structure.

Throughout this appendix, $\mathcal{K} \subseteq [k_{\min}, k_{\max}]$ denotes
the set of values of $k$ already evaluated, $\bar{z}(k)$ the monotone envelope
defined in \eqref{eq:envelope}, and
$\hat{k} = \max \mathcal{K}$ the largest value evaluated so far. Both baselines
share the initialisation of the RL-guided policy: the exact minimal feasible
deviation $k_{\min}$ is computed, the local branching model
$\mathcal{P}^{\mathrm{reop}}_m(k_{\min})$ is solved to obtain the anchor plan
with objective $z_0$, and the linear-programming floor $z_{\mathrm{LP}}$ is
computed once via the continuous relaxation at $k_{\max}$.

\subsection{Dichotomic search}
\label{app:ds}

The dichotomic search subdivides the intervals of the deviation axis where a
non-dominated plan has been shown to lie. Its behaviour is driven by the set
of \emph{active intervals}, defined as in Section~\ref{sec:rl_approach}. A pair
of consecutive evaluated values $k_1 < k_2$ with $k_2 - k_1 > 1$ is an active
interval if $\bar{z}(k_1) > \bar{z}(k_2)$: the envelope drops strictly between
$k_1$ and $k_2$, so a non-dominated plan may lie in $(k_1, k_2)$. When the
envelope does not drop, the pair is provably flat and is removed from
consideration once and for all.

At each iteration, the rule proceeds as follows. If at least one active
interval remains, the interval with the largest envelope drop
$\bar{z}(k_1) - \bar{z}(k_2)$ is selected and its midpoint
$k = \lfloor (k_1 + k_2)/2 \rfloor$ is evaluated. If no active interval remains
but the search has not yet reached the upper limit $k_{\max}$, the next value
is $\hat{k} + \delta_{\mathrm{s}}$, where the small step $\delta_{\mathrm{s}}$
is the same as in the RL policy, so that the sweep extends into the region
beyond $\hat{k}$. If neither operation applies -- no active interval and
$\hat{k} = k_{\max}$ -- the search terminates.

\begin{figure}[H]
\centering
\noindent\rule{\linewidth}{0.8pt}\\[-3pt]
\noindent\textbf{Algorithm 1.} Dichotomic search on the deviation radius.\\[-4pt]
\noindent\rule{\linewidth}{0.4pt}

\begin{tabbing}
xx\=xxxx\=xxxx\=xxxx\=\kill
1: \> Compute $k_{\min}$; solve $\mathcal{P}^{\mathrm{reop}}_m(k_{\min})$;
     set $\mathcal{K} \gets \{k_{\min}\}$, $\mathcal{F}^{*} \gets \{(k_{\min}, z_0)\}$ \\
2: \> \textbf{while} time budget not exhausted \textbf{do} \\
3: \> \> $\mathcal{I} \gets \{(k_1, k_2) : k_1, k_2 \in \mathcal{K},\
       k_2 - k_1 > 1,\ \bar{z}(k_1) > \bar{z}(k_2),\ (k_1,k_2)\ \text{not closed}\}$ \\
4: \> \> \textbf{if} $\mathcal{I} \neq \emptyset$ \textbf{then} \\
5: \> \> \> $(k_1^*, k_2^*) \gets \arg\max_{(k_1,k_2) \in \mathcal{I}}
       \bigl(\bar{z}(k_1) - \bar{z}(k_2)\bigr)$ \\
6: \> \> \> $k \gets \lfloor (k_1^* + k_2^*)/2 \rfloor$ \\
7: \> \> \textbf{else if} $\hat{k} < k_{\max}$ \textbf{then} \\
8: \> \> \> $k \gets \min\{\hat{k} + \delta_{\mathrm{s}},\, k_{\max}\}$ \\
9: \> \> \textbf{else} \textbf{stop} \\
10: \> \> \textbf{end if} \\
11: \> \> Solve $\mathcal{P}^{\mathrm{reop}}_m(k)$; update $\mathcal{K}$,
        $\bar{z}$, $\mathcal{F}^{*}$; close flat intervals \\
12: \> \textbf{end while} \\
13: \> \textbf{return} $\mathcal{F}^{*}$
\end{tabbing}

\vspace{-6pt}
\noindent\rule{\linewidth}{0.8pt}
\end{figure}

The dichotomic search is well-suited to localising non-dominated plans that
lie \emph{between} known Pareto points, since by construction each bisection
halves the range of values of $k$ in which such a plan may lie. Its structural
limitation is that it can only refine the region already bracketed by
$\mathcal{F}^{*}$: when the largest evaluated value $\hat{k}$ is smaller than
the location of an undiscovered non-dominated plan, no active interval brackets
that plan, and the rule advances by a small step $\delta_{\mathrm{s}}$ at a
time. Under a finite time budget, this behaviour prevents the search from
reaching plans that lie well beyond $\hat{k}$, which is the situation
illustrated in Figure~\ref{fig:hypervolume_illustration}.

\subsection{Fixed-$k$ grid}
\label{app:fixed_grid}

The fixed-$k$ grid evaluates $\mathcal{P}^{\mathrm{reop}}_m(k)$ at a set of
values of $k$ chosen \emph{a priori}, independently of the outcome of previous
solves. Given a target number of grid points $N_g$, the interval
$[k_{\min}, k_{\max}]$ is partitioned into $N_g$ equally spaced values
\begin{equation}
    k_i \;=\; k_{\min} + \left\lfloor i \cdot \frac{k_{\max} - k_{\min}}{N_g - 1} \right\rfloor,
    \qquad i = 0, 1, \ldots, N_g - 1,
    \label{eq:grid}
\end{equation}
which are evaluated in increasing order. In the experiments of
Section~\ref{sec:comparison}, $N_g = 10$; other values were investigated
during preliminary tuning and yielded comparable trends.

\begin{figure}[H]
\centering
\noindent\rule{\linewidth}{0.8pt}\\[-3pt]
\noindent\textbf{Algorithm 2.} Fixed-$k$ grid.\\[-4pt]
\noindent\rule{\linewidth}{0.4pt}

\begin{tabbing}
xx\=xxxx\=xxxx\=\kill
1: \> Compute $k_{\min}$, $k_{\max}$ \\
2: \> Build the grid $\{k_0, k_1, \ldots, k_{N_g - 1}\}$ according to
     Equation~\eqref{eq:grid} \\
3: \> $\mathcal{K} \gets \emptyset$, $\mathcal{F}^{*} \gets \emptyset$ \\
4: \> \textbf{for} $i = 0, 1, \ldots, N_g - 1$ \textbf{do} \\
5: \> \> \textbf{if} time budget exhausted \textbf{then} \textbf{stop} \textbf{end if} \\
6: \> \> Solve $\mathcal{P}^{\mathrm{reop}}_m(k_i)$; update $\mathcal{K}$,
        $\bar{z}$, $\mathcal{F}^{*}$ \\
7: \> \textbf{end for} \\
8: \> \textbf{return} $\mathcal{F}^{*}$
\end{tabbing}

\vspace{-6pt}
\noindent\rule{\linewidth}{0.8pt}
\end{figure}
The fixed-$k$ grid is the natural reference for the computational effort
required in the absence of any adaptive rule. Because it ignores the outcome
of previous solves, the same schedule of values is evaluated regardless of the
shape of the staircase $z^{*}(\cdot)$ induced by the disruption. Grid values
that fall on flat regions of the staircase consume solver calls without
extending the front, and closely spaced non-dominated plans between two grid
values are missed. In experiments, the resulting hit rate is systematically
lower than that of the dichotomic search or of the learned policy
(Table~\ref{tab:comparison}), which quantifies the value of adapting the
selection rule to the information revealed by previous solves.
\end{document}